%% file: btch-v02.tex
\documentclass[reqno,12pt]{amsart}
\usepackage{amsmath,amssymb,epsfig,color,slashbox,url}

\numberwithin{equation}{section}

\newtheorem{theorem}{Theorem}[section]
\newtheorem{conjecture}[theorem]{Conjecture}
\newtheorem{lemma}[theorem]{Lemma}
\newtheorem{proposition}[theorem]{Proposition}
\newtheorem{definition}[theorem]{Definition}
\newtheorem{corollary}[theorem]{Corollary}

\newtheorem{example}[theorem]{Example}
\newtheorem{remark}[theorem]{Remark}

\DeclareMathOperator{\link}{link}
\newcommand{\0}{\widehat{0}}
\newcommand{\1}{\widehat{1}}
\newcommand{\2}{\widehat{2}}
\newcommand{\nm}{\widehat{-1}}
\newcommand{\grI}{\widehat{I}}
\newcommand{\II}{I_2}
\newcommand{\grII}{\widehat{I_2}}
\newcommand{\iab}{\iota}

\newcommand{\Iab}{{\mathcal I}}
\newcommand{\IIab}{{\mathcal I}_2}

\newcommand{\lift}{{\mathcal L}}
\newcommand{\Sym}[1]{S_{\mathbb Q}\langle a,b\rangle_{#1}}
\newcommand{\Asym}[1]{A_{\mathbb Q}\langle a,b\rangle_{#1}}

\begin{document}

\title[The poset of intervals]{The type $B$ permutohedron and the
poset of intervals as a Tchebyshev transform}

\author[G\'abor Hetyei]{G\'abor Hetyei}

\address{Department of Mathematics and Statistics,
  UNC-Charlotte, Charlotte NC 28223-0001.
WWW: \tt http://webpages.uncc.edu/ghetyei/.}



\date{\today}

\begin{abstract}
We show that the order complex of intervals of a poset, ordered by
inclusion, is a Tchebyshev triangulation of the order complex of the
original poset. Besides studying the properties of this transformation,
we show that the dual of the type $B$ permutohedron is combinatorially
equivalent to the suspension of the order complex of the poset of
intervals of a Boolean algebra (with the minimum and maximum elements
removed).    
\end{abstract}
\maketitle

\section*{Introduction}

Inspired by Postnikov's seminal work~\cite{Postnikov}, we have seen a
surge in the study of root polytopes in recent years. A basic object in
these investigations is the permutohedron. This paper connects
permutohedra with a variant of the 
{\em Tchebyshev transform} of a poset, introduced by the present
author~\cite{Hetyei-tch,Hetyei-mfp} and studied by Ehrenborg and
Readdy~\cite{Ehrenborg-Readdy-Tch}, and with the (generalized)
{\em Tchebyshev triangulations} of a simplicial complex, first introduced by
the present author in~\cite{Hetyei-tt} and studied in collaboration with
Nevo in~\cite{Hetyei-Nevo}. The key idea of a Tchebyshev triangulation
may be summarized as follows: we add the midpoint to each edge of a
simplicial complex, and perform a sequence of stellar subdivisions,
until we obtain a triangulation containing all the newly added
vertices. Regardless of the order chosen, the face numbers of the
triangulation will be the same, and may be obtained from the face
numbers $f_j$ of the original complex by replacing the powers of $x$ with
Tchebyshev polynomials of the first kind if we work with the appropriate
generating function. The appropriate generating function in this setting
is the polynomial $F(x)=\sum_j f_{j-1} ((x-1)/2)^j$. It is also known that
the links of the original vertices in a Tchebyshev triangulation from a
multiset of simplicial complexes, called a {\em Tchebyshev triangulation
  of the second kind}, whose face numbers are also the same for all
Tchebyshev triangulations, and may be computed by replacing the powers
of $x$ with Tchebyshev polynomials of the second kind in $F(x)$.

The formula connecting the face numbers of the
type $A$ and type $B$ permutohedra is identical to computing the face
numbers of a Tchebyshev triangulation. These permutohedra are
simple polytopes, their duals are simplicial polytopes, their
boundary complexes are called the type $A$ resp.\ type $B$
  Coxeter complexes. The suspicion
arises that the type $B$ Coxeter complex is a Tchebyshev
triangulation of the type $A$ Coxeter complex. 

The present work contains the verification of this conjecture. The type
$A$ Coxeter complex is known to be the order complex of the Boolean
algebra, and the type $B$ Coxeter complex turns out to be the suspension
of an order complex, namely of the partially ordered set of intervals
of the Boolean algebra, ordered by inclusion.  We show that the operation of
associating the poset of intervals to a partially ordered sets 
always induces a Tchebyshev triangulation at the level of order
complexes. This observation may be helpful in constructing ``type $B$
analogues'' of other polytopes and partially ordered sets. Furthermore
it inspires further study of the poset of intervals of a poset,
initiated by Walker~\cite{Walker}, and continued by
Athanasiadis~\cite{Athanasiadis}, Athanasiadis and
Savvidou~\cite{Athanasiadis-Savvidou} and
Joji\'{c}~\cite{Jojic} among others. 

This paper is structured as follows. After the Preliminaries, we
introduce the poset of intervals in Section~\ref{sec:itcheb} and show
that the order complex of the poset of intervals is always a Tchebyshev
triangulation of the order complex of the original poset. We also introduce
a graded variant of this operation that takes a graded poset into a
graded poset. In Section~\ref{sec:typeb} we show that the type $B$
Coxeter complex is the order complex of the graded poset of intervals of the
Boolean algebra. In Section~\ref{sec:flagf} we review how to compute the
flag $f$-vector of graded a poset of intervals. This topic was first
studied by Joji\'{c}~\cite{Jojic}, and we provide new proofs to some of
his key formulas. Section~\ref{sec:t2} introduces interval transforms of
the second kind. The corresponding multiset of order complexes is the
Tchebyshev triangulation of the second kind corresponding to the
Tchebyshev triangulation induced by taking the order complex of the
graded poset of intervals of a graded poset. We find explicit flag
$f$-vector formulas in terms of the mixing operator introduced by  
Ehrenborg and Readdy~\cite{Ehrenborg-Readdy-cop}. Inspired by the work
of Ehrenborg and Readdy~\cite{Ehrenborg-Readdy-Tch}, we make the first
steps towards describing all eigenvectors of the linear operator on the
flag $f$-vectors, induced by taking the interval transforms of the
second kind. In Section~\ref{sec:Eulerian} we consider the special case
of Eulerian posets, cite a formula by Joji\'{c}~\cite{Jojic} and an
analogous recurrence found by Ehrenborg and Fox~\cite{Ehrenborg-Fox} for
the mixing operator, which may be used to compute the effect on the
$cd$-index of taking the interval transform of the second kind. 
The latter result is used in Section~\ref{sec:special} to compute the
$cd$-index of the interval transform of the second kind of the ladder
poset (the same calculation was already performed by
Joji\'{c}~\cite{Jojic} for the interval transform of the first kind of
these posets). As part of the proof of our formula, we develop a
weighted lattice path enumeration model to express the values
$M(c^i,c^j)$ for the mixing operator of Ehrenborg and
Readdy~\cite{Ehrenborg-Readdy-cop}. The other special example considered
in this section is the Boolean lattice, where known results of
Purtill~\cite{Purtill}, Hetyei~\cite{Hetyei-andre} and of Ehrenborg and
Readdy~\cite{Ehrenborg-Readdy-rcubical} come into play. These results
use Andr\'e permutations, first studied by Foata, Strehl and 
Sch\"utzenberger~\cite{Foata-Schutzenberger,Foata-Strehl}, and their
signed generalizations. 

\section{Preliminaries}

\subsection{Graded Eulerian posets}

A partially ordered set is
{\em graded} if it contains a unique minimum element $\0$, a unique
maximum element $\1$ and a rank function $\rho$ satisfying $\rho(\0)=0$
and $\rho(y)=\rho(x)+1$ for each $x$ and $y$ such that $y$ covers $x$.
The number of chains containing elements of  fixed sets of ranks in 
a graded poset $P$ of rank $n+1$  is encoded by the {\em flag
  $f$-vector} $(f_S(P)\::\: S\subseteq \{1,\ldots,n\})$. The entry $f_S$
in the flag $f$-vector  is the number of chains
$x_1<x_2<\cdots<x_{|S|}$ such that their set of ranks $\{\rho(x_i)\::\:
i\in \{1,\ldots,|S|\}\}$ is $S$. Inspired by Stanley~\cite{Stanley-flag} we introduce
the {\em upsilon invariant} of a graded poset $P$ of rank $n+1$ by 
$$
\Upsilon_P(a,b)=\sum_{S\subseteq \{1,\ldots,n\}} f_S u_S
$$
where $u_S=u_1\cdots u_n$ is a monomial in noncommuting variables $a$
and $b$ such that $u_i=b$ for all $i\in S$ and $u_i=a$ for all $i\not\in
S$. It should be noted that the term upsilon invariant is {\em not used
elsewhere} in the literature, most sources switch to the {\em $ab$-index
  $\Psi_P(a,b)$} defined as $\Upsilon_P(a-b,b)$. The
$ab$-index may be also written as a linear combination of monomials in
$a$ and $b$, the coefficients of these monomials form the {\em flag $h$-vector}.
A graded poset $P$ is {\em Eulerian} if every nontrivial interval of $P$
has the same number of elements of even rank as of odd rank. All linear
relations satisfied by the flag $f$-vectors of Eulerian posets were
found by Bayer and Billera~\cite{Bayer-Billera}. A very useful and
compact rephrasing of the Bayer--Billera relations was given by Bayer and
Klapper in~\cite{Bayer-Klapper}: they proved that satisfying the
Bayer--Billera relations is equivalent to stating that the $ab$-index may
be rewritten as a polynomial of $c=a+b$ and $d=ab+ba$. The resulting
polynomial in noncommuting variables $c$ and $d$ is called the {\em
  $cd$-index}.  

As an immediate consequence of the above cited results we obtain the
following.
\begin{corollary}
The $cd$-index of a graded Eulerian poset $P$ may be obtained by
rewriting $\Upsilon_P(a,b)$ as a polynomial of $c=a+2b$ and $d=ab+ba+2b^2$.
\end{corollary}  
Note that this statement is a direct consequence of
$\Upsilon_P(a-b,b)=\Psi_P(a,b)$ which is equivalent to
$\Upsilon_P(a,b)=\Psi_P(a+b,b)$.

\subsection{Tchebyshev triangulations and Tchebyshev transforms}

A finite simplicial complex $\triangle$ is a family of subsets of a
finite vertex set $V$. The elements of $\triangle$ are called {\em
  faces}, subject to the following rules: a subset of any face is a face
and every singleton is a face. The {\em dimension} of a face is one less than
the number of its elements, the dimension $d-1$ of the
complex $\triangle$ is the maximum of the dimension of its faces. The
number of $j$-dimensional faces is denoted by $f_j(\triangle)$ and the
vector $(f_{-1}, f_0,\ldots,f_{d-1})$ is the {\em $f$-vector} of the
simplicial complex. We define the {\em $F$-polynomial $F_{\triangle}(x)$
} of a finite simplicial complex $\triangle$ as
\begin{align}
F_{\triangle}(x)=\sum_{j=0}^{d} f_{j-1}(\triangle)\cdot
\left(\frac{x-1}{2}\right)^j. 
\end{align}
The {\em join $\triangle_1*\triangle_2$} of two simplicial
complexes $\triangle_1$ and $\triangle_2$ on disjoint vertex sets is the
simplicial complex  
$\triangle_1*\triangle_2=\{\sigma\cup\tau\::\: \sigma\in\triangle_1,
\tau\in\triangle_2\}$. It is easy to show that the $F$-polynomials
satisfy $F_{\triangle_1*\triangle_2}(x)=F_{\triangle_1}(x)\cdot
F_{\triangle_2}(x)$. A special instance of the join operation is the
{\em suspension} operation: the suspension $\triangle*\partial(\triangle^1)$ 
of a simplicial complex $\triangle$ is the join of $\triangle$ with the
boundary complex of the one dimensional simplex. (A $(d-1)$-dimensional
simplex is the family of all subsets of a $d$-element set, its boundary
is obtained by removing its only facet from the list of faces.)
The {\em link} of a face $\sigma$
is the subcomplex $\link_{\triangle}(\sigma)=\{\tau\in \triangle:\ 
\sigma\cap \tau=\emptyset, \ \sigma\cup \tau\in K\}$. A special type of
simplicial complex we will focus on is the {\em order complex
  $\triangle(P)$} of a finite partially ordered set $P$: its vertices
are the elements of $P$ and its faces are the increasing chains. The
order complex of a finite poset is a {\em flag complex}: its minimal
non-faces are all two-element sets (these are the pairs of incomparable
elements). 

Every finite simplicial complex $\triangle$
has a {\em standard geometric realization} in the vector space with
a basis $\{e_v\::\: v\in V\}$ indexed by the vertices, where each face
$\sigma$ is realized by the convex hull of the basis vectors $e_v$
indexed by the elements of $\sigma$. 
\begin{definition}
We define a {\em Tchebyshev triangulation} $T(\triangle)$ of a finite
simplicial complex $\triangle$ as follows. We number the edges $e_1,
e_2, \ldots, e_{f_1(\triangle)}$ in some order, and we associate to each
edge $e_i=\{u_i,v_i\}$ a midpoint $w_i$. We associate a sequence
$\triangle_0:=\triangle,\triangle_{1},\triangle_{2}\ldots,\triangle_{f_1(\triangle)}$
of simplicial complexes to this numbering of edges, as
follows. For each $i\geq 1$, the complex $\triangle_{i}$ is obtained
from $\triangle_{i-1}$ by replacing the edge $e_i$ and the faces
contained therein with the one-dimensional simplicial complex $L_i$,
consisting of the vertex set $\{u_i,v_i,w_i\}$ and edge set
$\{\{u_i,w_i\}, \{w_i,v_i\}\}$,  and by replacing the family of faces
$\{e_i\cup\tau\:: \tau\in \link_{\Delta_{i-1}}(e_i)\}$ containing $e_i$
with the family of faces $\{\sigma'\cup \tau \:: \sigma'\in L_i\}$. In
other words, we subdivide the edge $e_i$ into a path of length $2$ by
adding the midpoint $w_i$ and we also subdivide all faces containing
$e_i$, by performing a stellar subdivision. 
\end{definition}  
As it is defined by a sequence of a stellar subdivisions, it is clear
that any Tchebyshev triangulation of $\triangle$ as defined above is indeed a
triangulation of $\triangle$ in the following sense: if we consider the
standard geometric realization of $\triangle$ and associate to each
midpoint $w$ the midpoint of the line segment realizing the
corresponding edge $\{u,v\}$ then the convex hulls of the vertex sets
representing the faces of $T(\triangle)$ represent a triangulation of
the geometric realization of $\triangle$. Furthermore, the
following statement is a special case of~\cite[Theorem 3.3]{Hetyei-Nevo}
and can also be derived from~\cite[Example 2.8]{Athanasiadis-survey}, combined
with Stanley's locality formula~\cite[Theorem 3.2]{Stanley-subdivisions}.

\begin{theorem}
\label{thm:Hetyei-Nevo}  
All Tchebyshev triangulations of a simplicial complex have the same $f$-vector.
\end{theorem}
\begin{remark}
{\em   
\cite[Theorem 3.3]{Hetyei-Nevo} allows replacing the operation of
taking the midpoint of each edge with higher dimensional analogues. On
the other hand, every Tchebyshev triangulation of $\triangle$ dissects
each $k$-dimensional face into exactly $2^k$ faces of the same dimension. This
property is shared by other triangulations of $\triangle$, such as the second
edgewise subdivision, introduced by Freudenthal~\cite{Freudenthal}. All
triangulations of $\triangle$ with this property have the same
$f$-vector by~\cite[Example 2.8]{Athanasiadis-survey}, 
combined with~\cite[Theorem 3.2]{Stanley-subdivisions}. Thus, the
formulas obtained by Brenti 
and Welker~\cite{Brenti-Welker} for the $h$-vector of the second
edgewise subdivision of $\triangle$ 
apply to Tchebyshev triangulations as well. See also
Remark~\ref{rem:second} below.  
} 
\end{remark}  

Tchebyshev triangulations of the second kind were first introduced
in~\cite{Hetyei-tt} in connection with some special Tchebyshev
triangulations of the second kind. The idea was generalized to arbitrary
generalized Tchebyshev triangulations in~\cite{Hetyei-Nevo}. Here we
specialize the definition introduced in~\cite{Hetyei-Nevo} to Tchebyshev
triangulations as follows. Recall that the {\em link
  $\link_{\triangle}(\tau)$} of a 
face $\tau$ in a simplicial complex $\triangle$ is the set of faces
$\{\sigma-\tau\: : \: \sigma\in \triangle, \tau\subseteq \sigma\}$. 
\begin{definition}
Let $\triangle$ be an arbitrary simplicial complex with vertex set $V$
and $T(\triangle)$ a Tchebyshev triangulation. We define the {\em
  corresponding Tchebyshev triangulation of the second kind
  $U(\triangle)$} as the collection of the links
$\link_{T(\triangle)}(\{v\})$ for all vertices $v\in V$.
\end{definition}  
Note that $U(\triangle)$ is not a simplicial complex, but a multiset of
simplicial complexes. We define its $f$-vector ($F$-polynomial) as the
sum of the $f$-vectors ($F$-polynomials) of the complexes
$\link_{T(\triangle)}(\{v\})$ for all $v\in V$. The following result is
a direct consequence of~\cite[Theorem 3.3]{Hetyei-Nevo}. 
\begin{theorem}[Hetyei and Nevo]
\label{thm:Hetyei-Nevo2}  
All Tchebyshev triangulations of the second kind of a simplicial complex
have the same $f$-vector. 
\end{theorem}
\begin{figure}[h]
  \input{triangulations_pdf.tex} 
\caption{A Tchebyshev triangulation and a second edgewise triangulation}  
\label{fig:triangulations}
\end{figure}
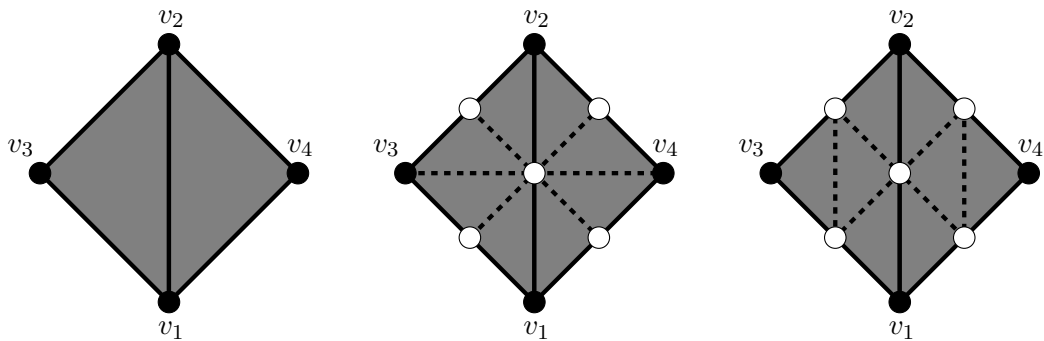  
\begin{remark}
{\em While Tchebyshev triangulations have the same face numbers as the
  second edgewise triangulation, this result cannot be extended to
  Tchebyshev triangulations of the second
  kind. Figure~\ref{fig:triangulations} shows a simplicial complex with
  $4$ (black) vertices, $5$ edges and $2$ two-dimensional faces. A
  Tchebyshev triangulation (shown in the middle, obtained by performing
  the first stellar subdivision at the midpoint of the edge
  $\{v_1,v_2\}$) has the same face   numbers as the second edgewise
  triangulation (on the right). However, the sum of the $f$-vectors of the
  links of the original vertices is different in the two triangulations.
}
\label{rem:second}
\end{remark}

The following result has been shown in~\cite[Propositions 3.3 and
  4.4]{Hetyei-tt} for a specific Tchebyshev triangulation. By the
preceding theorems it holds for all Tchebyshev triangulations and
motivates the choice of the terminology. The {\em Tchebyshev transform
  $T$ ($U$) of the first (second) kind of polynomials} used in the
next result is  the linear map ${\mathbb 
  R}[x]\longrightarrow {\mathbb R}[x]$ sending $x^n$ into the Tchebyshev
polynomial of the first kind $T_n(x)$ (second kind $U_n(x)$).
\begin{theorem}
  \label{thm:tt}
For any finite simplicial complex $\triangle$, the $F$-polynomial of any
Tchebyshev triangulation $T(\triangle)$ is the Tchebyshev transform of
the first kind of the $F$-polynomial of $\triangle$:
$$
F_{T(\triangle)}(x)=T(F_{\triangle}(x)).
$$
Similarly, the $F$-polynomial of any Tchebyshev triangulation
$U(\triangle)$ of the second kind is half of the Tchebyshev transform of
the second kind of the $F$-polynomial of $\triangle$:
$$
F_{U(\triangle)}(x)=\frac{1}{2}\cdot U(F_{\triangle}(x)).
$$
\end{theorem}

The notion of the Tchebyshev triangulation of a simplicial complex was
motivated by a poset operation, first considered in~\cite{Hetyei-tch}
and formally introduced in~\cite{Hetyei-mfp} .
\begin{definition}
Given a locally finite poset $P$, its {\em Tchebyshev transform of the
  first kind $T(P)$} is
the poset whose elements are the intervals $[x,y]\subset P$ satisfying
$x\neq y$, ordered by the following relation: $[x_1,y_1]\leq [x_2,y_2]$
if either $y_1\leq x_2$ or both $x_1=x_2$ and $y_1\leq y_2$ hold.
\end{definition}
A geometric interpretation of this operation may be found
in~\cite[Theorem 1.10]{Hetyei-mfp}. The graded variant of this poset
operation is defined in~\cite{Ehrenborg-Readdy-Tch}.  
Given a graded  poset $P$ with minimum element $\0$ and maximum element
$\1$, we introduce a new minimum element $\nm<\0$  and a new maximum
element $\2$. The {\em graded Tchebyshev transform of the first kind} of
a graded poset $P$ is then the interval $[(\nm,\0),(\1,\2)]$ in
$T(P\cup\{\nm,\2\})$.  By abuse of notation we also denote the graded
Tchebyshev transform of a graded poset $P$ by $T(P)$. It is easy to show
that $T(P)$ is also a graded poset, whose rank is one more than that of $P$.
The following result may be found in~\cite[Theorem 1.5]{Hetyei-tt}.
\begin{theorem}
\label{thm:tpgraded}
Let $P$ be a graded poset and $T(P)$ its graded Tchebyshev transform.
Then the order complex $\triangle(T(P)\setminus \{(\nm,\0), (\1,\2)\})$
is a Tchebyshev triangulation of the suspension of $\triangle(P\setminus
\{\0,\1\})$. 
\end{theorem}  
As a consequence of Theorem~\ref{thm:tpgraded}, we have
\begin{align}
\label{eq:tposet}  
  F_{\triangle(T(P)\setminus \{(\nm,\0), (\1,\2)\})}
&=T(x\cdot F_{\triangle(P\setminus \{\0,\1\})}). 
\end{align}
It has been shown by Ehrenborg and
Readdy~\cite{Ehrenborg-Readdy-Tch} that there is a linear transformation
assigning to the flag $f$-vectors of each graded poset $P$ of rank $n+1$
the flag $f$-vector of its Tchebyshev transform of the first kind
$T(P)$. For Eulerian posets, they also compute the effect on the
$cd$-index of taking the Tchebyshev transform of the first kind.      
They also studied the corresponding Tchebyshev transforms of the second
kind.

\subsection{Permutohedra of type $A$ and $B$}
Permutohedra of type $A$ and $B$ have a vast literature, the results
cited here may be found in~\cite{Fomin-Reading} and
in~\cite{Wachs-pt}.  

The type $A$ permutohedron $\operatorname{Perm}(A_{n-1})$ is the convex
hull of the $n!$ vertices\\ $(\pi(1),\ldots,\pi(n))\in\mathbb{R}^n$, where
$\pi$ is any permutation of the set
$[1,n]:=\{1,2,\ldots,n\}$. The type $B$ permutohedron
$\operatorname{Perm}(B_{n})$ is the convex hull of all points of the form
$(\pm \pi(1),\pm
\pi(2)\ldots,\pm\pi(n))\in\mathbb{R}^n$. Combinatorially equivalent
polytopes may be obtained by taking the $A_{n-1}$-orbit, respectively
$B_n$ orbit, of any sufficiently generic point in an $(n-1)$-dimensional
(respectively $n$-dimensional) space, and the convex hull of the points
in the orbit.~\cite[Section 2]{Fomin-Reading}.

The type $A$ and $B$ permutohedra are simple polytopes, their duals are
simplicial polytopes. The boundary complexes of these duals are
combinatorially equivalent to the {\em Coxeter complexes} of 
the respective Coxeter groups. The Coxeter complex of the symmetric
group $A_{n-1}$ on $[1,n]$ is the order complex of
$P([1,n])-\{\emptyset,[1,n]\})$, where $P([1,n])$ is the Boolean algebra
of rank $n$. The Coxeter complex of the Coxeter group $B_n$ is the
order complex of the face lattice of the $n$-dimensional
crosspolytope~\cite[Lecture 1]{Wachs-pt}. In either case we consider
the order complexes of the respective graded posets without their unique
minimum and maximum elements: adding these would make the order complex
contractible, whereas the boundary complexes of simplicial polytopes are
homeomorphic to spheres. The standard $n$-dimensional crosspolytope
is the convex hull of the vertices $\{\pm e_i\::\: i\in [1,n]\}$, where
$\{e_1,e_2,\ldots,e_n\}$ is the standard basis of ${\mathbb R^n}$. Each
nontrivial face of the crosspolytope is the convex hull of a set of
vertices of the form $\{e_i, i\in K^+\}\cup \{-e_i, i\in K^-\}$, where
$K^+$ and $K^-$ is are disjoint subsets of $[1,n]$ and their union is
not empty. Keeping in mind that each face of a polytope is the
intersection of all the facets containing it, we have the following
consequence.

\begin{corollary}
\label{cor:facesintervals}  
Each facet of $\operatorname{Perm}(B_{n})$ is uniquely labeled with a
pair of sets $(K^+,K^-)$ where $K^+$ and $K^-$ is are subsets
of $[1,n]$, satisfying $K^+\subseteq [1,n]-K^-$ and $K^+$ and $K^-$
cannot be both empty. For a set of valid labels
  $$\{(K_1^+,K_1^-),(K_2^+,K_2^-),\ldots, (K_m^+,K_m^-)\}$$
the intersection of the corresponding set of
facets is a nonempty face of $\operatorname{Perm}(B_{n})$ if and only if 
$$
K_1^+\subseteq K_2^+\subseteq \cdots \subseteq K_m^+ \subseteq
[1,n]-K_m^- \subseteq [1,n]-K_{m-1}^- \subseteq \cdots 
\subseteq [1,n]- K_1^-\quad\mbox{holds.}   
$$
\end{corollary}

The triangle of $f$-vectors of the type $B$ Coxeter complexes is
given in sequence A145901 in~\cite{OEIS}.

\section{The poset of intervals as a Tchebyshev transform}
\label{sec:itcheb}

We will use Corollary~\ref{cor:facesintervals} to represent the type $B$
Coxeter complex using the {\em poset of intervals} of a Boolean
algebra. In this section we review this construction and show that
taking the poset of intervals induces a Tchebyshev triangulation. 

\begin{definition}
An {\em interval $[u,v]$} in a partially ordered set $P$ is the set of
all elements $w\in P$ satisfying $u\leq w\leq v$.    
For a finite partially ordered set $P$ we define the {\em poset $I(P)$
  of the intervals of $P$} as the set of all intervals $[u,v]\subseteq
  P$, ordered by inclusion.
\end{definition}  

We may identify the singleton intervals $[u,u]$ in $I(P)$ with the
elements of $P$. This subset of elements forms an antichain in $I(P)$,
however, under this identification, the order complex of $I(P)$ looks
like a triangulation of the order complex of $P$, see
Figures~\ref{fig:poset}  and \ref{fig:posett}. Figure~\ref{fig:poset}
shows a partially ordered set and its order complex. The poset of its
intervals and the order complex thereof may be seen in
Figure~\ref{fig:posett}.  

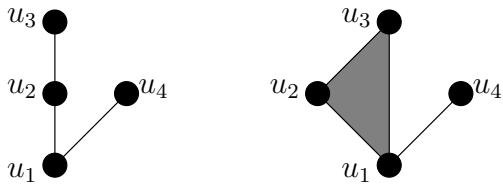
\begin{figure}[h]
  \input{poset_pdf.tex}
\caption{A partially ordered set $P$ and its order complex $\triangle(P)$}  
\label{fig:poset}
\end{figure}  

In Figure~\ref{fig:posett} we marked the vertices of the order complex
associated to non-singleton intervals with white circles.

\begin{figure}[h]
  \input{posett_pdf.tex}
\caption{The poset $I(P)$ of intervals of $P$ and its order complex}  
\label{fig:posett}
\end{figure}
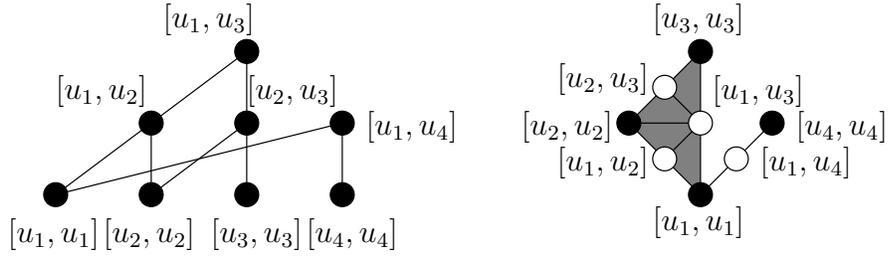  

The following result is a generalization of~\cite[Remark 10]{Jojic}, and
an equivalent restatement of Walker's result~\cite[Theorem 4.1]{Walker}.

\begin{theorem}
\label{thm:it}
For any finite partially ordered set $P$ the order complex
$\triangle(I(P))$ of its poset of intervals is isomorphic to a
Tchebyshev triangulation of $\triangle(P)$ as follows. For each $u\in
P$ we identify the vertex $[u,u]\in \triangle(I(P))$ with the vertex $u\in
\triangle(P)$ and for each nonsingleton interval $[u,v]\in I(P)$ we
identify the vertex $[u,v]\in \triangle(I(P))$ with the midpoint of the
edge $\{[u,u],[v,v]\}$. We number the midpoints $[u_1,v_1], [u_2,v_2], \ldots$  
in such an order that $i<j$ holds whenever the interval $[u_i,v_i]$
contains the interval $[u_j,v_j]$.
\end{theorem}  
\begin{proof}
We illustrate the Tchebyshev triangulation process with the poset shown
in Figure~\ref{fig:posett}. We list its nonsingleton intervals in the
following order:
$[u_1,u_3]$, $[u_1,u_2]$, $[u_2,u_3]$, $[u_1,u_4]$. Figure~\ref{fig:posettpart}
shows the stage of the process when we already added $[u_1,u_3]$ and
$[u_1,u_2]$ but none of the remaining nonsingleton intervals.
\begin{figure}[h]
  \input{posettpart_pdf.tex}
\caption{The second step of the Tchebyshev triangulation process}  
\label{fig:posettpart}
\end{figure}
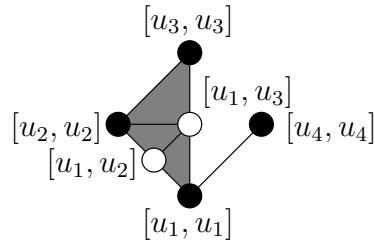
The following statement may be shown by induction on the number of stages
in the process: in each stage, the resulting complex is a flag complex,
whose minimal nonfaces are the following:
\begin{enumerate}
\item Pairs of singletons $\{[u,u],[v,v]\}$ such that $u$ and $v$ are
  not comparable in $P$.
\item Pairs of singletons $\{[u,u],[v,v]\}$ such that $u<v$ holds in
  $P$, but the interval $[u,v]$ has already been added to the triangulation.
\item Pairs of intervals from $I(P)$ such that neither one contains the other.
\end{enumerate}  
In each stage of the process, the nonsingleton interval $[u,v]$
added is the first midpoint of any edge whose endpoints are contained in
the interval $[u,v]$ of $P$. At the beginning of the stage the
restriction of the current complex to intervals contained in  $[u,v]$
only contains singleton intervals, and it is isomorphic to the order
complex of $[u,v]$. Subdividing the edge $\{[u,u],[v,v]\}$ and
all faces containing this edge results in a complex where both $[u,u]$
and $[v,v]$ can not appear in the same face any more, each such face is
replaced with $2$ faces: one containing $\{[u,u], [u,v]\}$  the other
containing $\{[v,v], [u,v]\}$. All intervals $[u',v']$ containing
$[u,v]$ have already been added in a previous stage, and now we add the
edge $\{[u,v],[u',v']\}$. The cumulative effect of all these changes is
that we obtain a new flag complex satisfying the listed criteria. 
\end{proof}
\begin{remark}
{\em Walker's proof is a direct geometric argument. The proof above uses
  the more general result stated in~\cite[Theorem 3.3]{Hetyei-Nevo}. It also
  directly implies the face counting formula that holds for all
  Tchebyshev transforms.}
\end{remark}

When $P$ is a graded poset then $[u',v']$ covers $[u,v]$ in $I(P)$
exactly when the rank function $\rho$ of $P$ satisfies 
$\rho(v')-\rho(u')=\rho(v)-\rho(u)+1$. Hence we may define the following
graded variant of the operation $P\mapsto I(P)$.

\begin{definition}
For a graded poset $P$ we define its {\em graded poset of intervals
  $\grI(P)$} as the poset of all intervals of $P$, {\em including the
    empty set}, ordered by inclusion.
\end{definition}  

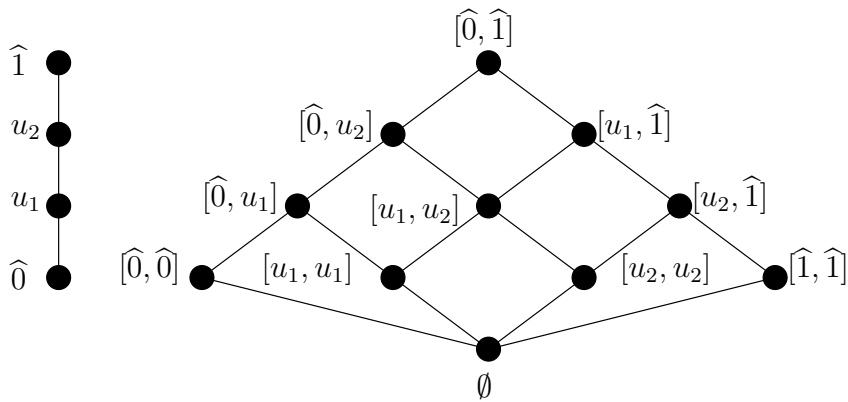
\begin{figure}[h]
  \input{gradedi_pdf.tex}
\caption{The graded poset of intervals of a chain}  
\label{fig:gradedi}
\end{figure}

\begin{remark}
{\em Figure~\ref{fig:gradedi} represents the graded poset of intervals
  of a chain of rank $3$. It is worth comparing this illustration
  with~\cite[Figure 2]{Hetyei-mfp} where the Tchebyshev transform of a
  chain of rank $3$ is represented. The two posets are not isomorphic,
  not even after taking the dual of the Tchebyshev transform to make the
number of elements at the same rank equal.}
\end{remark}  

The following statement is straightforward.
\begin{proposition}
If $P$ is a graded poset of rank $n$ with rank function $\rho$ then
$\grI(P)$ is a graded poset of rank $n+1$, in which the rank of a
nonempty interval $[u,v]$ is $\rho(v)-\rho(u)+1$.
\end{proposition}  

In analogy to Theorem~\ref{thm:tpgraded} we have the following result.

\begin{proposition}
\label{prop:ipgraded}
Let $P$ be a graded poset and $\grI(P)$ its graded poset of intervals.
Then the order complex $\triangle(\grI(P)- \{\emptyset, [\0,\1]\})$
is a Tchebyshev triangulation of the suspension of $\triangle(P-\{\0,\1\})$. 
\end{proposition}
\begin{proof}
By Theorem~\ref{thm:it}, the order complex
$\triangle(\grI(P)-\{\emptyset\})$ is a Tchebyshev triangulation of
$\triangle(P)$. The order complex $\triangle (P)$ is the join of
$\triangle(P-\{\0,\1\})$ with the one-dimensional simplex on the vertex
set $\{\0,\1\}$. Performing the Tchebyshev triangulation results in
subdividing every simplex containing the edge $\{\0,\1\}$ into two
simplices. The removal of the midpoint $[\0,\1]$ leaves us exactly with
those faces which are contained in a face of $\triangle (P)$ that does not
contain the edge $\{\0,\1\}$. Hence we obtain a Tchebyshev triangulation
of a suspension of $\triangle(P-\{\0,\1\})$: the suspending vertices are
$\0$ and $\1$.
\end{proof}  

We conclude this section with the following observations regarding the
{\em direct product} of two graded posets. Recall that the direct
product $P\times Q$ of two graded posets $P$ and $Q$ is defined as the
set of all ordered pairs $(u,v)$ where $u\in P$ and $v\in Q$, subject to
the partial order $(u_1,v_1)\leq (u_2,v_2)$ holding exactly when
$u_1\leq u_2$ holds in $P$ and $v_1\leq v_2$ holds in $Q$.

\begin{proposition}
\label{prop:cprod}  
If $P$ and $Q$ are graded posets then $I(P\times Q)$ is isomorphic to
$I(P)\times I(Q)$.
\end{proposition}
The straightforward verification is left to the reader.
Proposition~\ref{prop:cprod} may be immediately generalized to the 
graded poset of intervals using the   {\em diamond product} introduced
by Ehrenborg and Readdy in~\cite{Ehrenborg-Readdy-Tch}.
\begin{definition}
Given two graded posets $P$ and $Q$, their {\em diamond product}
$P\diamond Q$ is defined as $(P-\{\0_P\})\times (Q-\{\0_Q\})\cup \0$. In
  other words, to obtain the diamond product we remove the unique
  minimum elements of $P$ and $Q$ respectively, we take the direct
  product of the resulting posets and we add a new unique minimum element. 
\end{definition}
\begin{corollary}
\label{cor:cprod}    
If $P$ and $Q$ are graded posets then $\grI(P\times Q)$ is isomorphic to
$\grI(P)\diamond \grI(Q)$.
\end{corollary}  
A special case of Corollary~\ref{cor:cprod} may be found
in~\cite[Proposition 4 (iv)]{Jojic}.

\begin{remark}
{\em It is worth comparing Corollary~\ref{cor:cprod} above with
  ~\cite[Theorem 9.1]{Ehrenborg-Readdy-Tch} where it is stated that the
  Tchebyshev transform of the Cartesian product of two posets is the
  diamond product of their Tchebyshev transforms.}
\end{remark}

\section{The type $B$ Coxeter complex as a Tchebyshev triangulation}
\label{sec:typeb}

After introducing $X:=K^+$ and $Y:=[1,n]-K^-$, we may rephrase
Corollary~\ref{cor:facesintervals} as follows.   
\begin{corollary}
We may label each facet of the type $B$ permutohedron
$\operatorname{Perm}(B_{n})$ with a nonempty interval $[X,Y]$ of the
Boolean algebra $P([1,n])$ that is different from $P([1,n])=[\emptyset, [1,n]]$. 
The set $\{[X_1,Y_1], [X_2,Y_2], \ldots, [X_m,Y_m]\}$ labels a
collection of facets with a nonempty intersection if and only if the
intervals form an increasing chain in $\grI(P([1,n]))-\{\emptyset,
[\emptyset, [1,n]]\}$.   
\end{corollary}  

The representation of each face of
$\operatorname{Perm}(B_{n})$ as an intersection of facets is
unique, hence we obtain the following result.

\begin{proposition}
The dual of $\operatorname{Perm}(B_{n})$ is a simplicial polytope whose
boundary complex is combinatorially equivalent to the order complex
$\triangle(\grI(P([1,n]))-\{\emptyset, [\emptyset, [1,n]]\})$.
\label{prop:tbc}
\end{proposition}  

As a consequence of this statement and of
Proposition~\ref{prop:ipgraded}, we obtain the following result. 
\begin{corollary}
\label{cor:typeb-tt}  
The dual of $\operatorname{Perm}(B_{n})$ is a simplicial polytope whose
boundary complex is combinatorially equivalent to a Tchebyshev
triangulation of the suspension of $\triangle(P([1,n])-\{\emptyset, [1,n]\})$.
\end{corollary}
The order complex $\triangle(\grI(P([1,n]))-\{\emptyset, [\emptyset,
  [1,n]]\})$ has been studied by
Athanasiadis and Savvidou~\cite{Athanasiadis-Savvidou}.  It is worth
noting that the order complex $\triangle(P([1,n])-\{\emptyset, [1,n]\})$ 
is known to be combinatorially equivalent to the dual of the boundary
complex of the permutohedron $\operatorname{Perm}(A_{n-1})$. We may also
think of this complex as the barycentric subdivision of the boundary of an
$(n-1)$-dimensional simplex. 

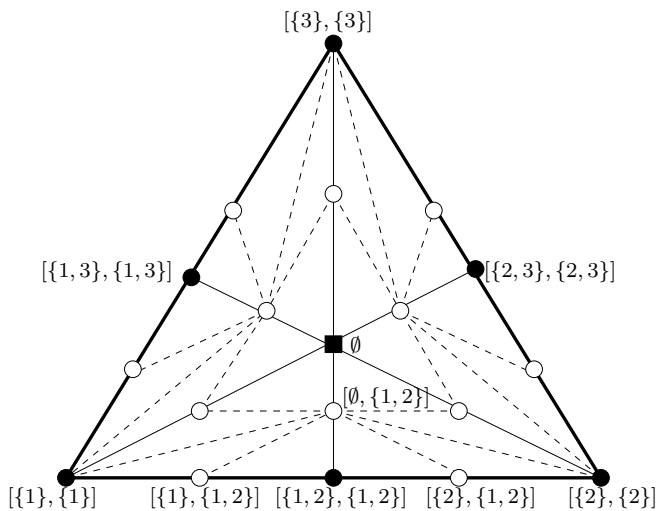
\begin{figure}[h]
  \input{typebdual_pdf.tex}
\caption{Half of the dual of $\operatorname{Perm}(B_{3})$}  
\label{fig:typebdual}
\end{figure}
Figure~\ref{fig:typebdual} represents ``half'' of the dual of
$\operatorname{Perm}(B_{3})$. The boundary of the triangle whose
vertices are labeled with singleton intervals $[\{i\},\{i\}]$ is shown
in bold. (In general, the reader should imagine the boundary of a
simplex, whose vertices are labeled with $[\{i\},\{i\}]$.) The vertices
of the barycentric subdivision of the boundary are marked with black
circles. These correspond to singleton intervals of the form $[X,X]$,
where $X$ is a subset of $[1,3]$. (In general, $X$ is a subset of $[1,n]$.)
The suspending vertex $\emptyset$ is marked with a black square. The
other suspending vertex $[1,3]$ (in general: $[1,n]$) is not shown in the
picture. One would need to make another picture showing the boundary of
the triangle with the suspending vertex, and ``glue'' the two pictures
along the boundary of the triangle. The midpoints of the edges are
marked with white circles. These are labeled with intervals $[X,Y]$ such
that $X$ is properly contained in $Y$. The edges arising when we take
the appropriate Tchebyshev triangulation are indicated with dashed
lines. Note that this part of the picture is different on the ``other
side'' of the dual of $\operatorname{Perm}(B_{3})$: on the side shown
the largest intervals labeling midpoints are of the form $[\emptyset,
  [1,3]-\{i\}]$ (in general $[\emptyset, [1,n]-\{i\}]$) whereas on the other
side the largest such intervals are of the form $[\{i\}, [1,3]]$ (in
general: $[\{i\}, [1,n]]$). We leave to the reader as a challenge to draw
the other side of the dual of $\operatorname{Perm}(B_{3})$.     

\begin{remark}
{\em By Proposition~\ref{prop:tbc}, the work of Anwar and
  Nazir~\cite{Anwar} implies that the $h$-polynomial  of the type
  $B$ Coxeter complex has only real roots. As a consequence of
  Corollary~\ref{cor:typeb-tt} we know that this is a Tchebyshev
  triangulation and we may compute its $F$-polynomial
  using~\eqref{eq:tposet}, and obtain that these 
polynomials have the same coefficients (up to sign) as the derivative
polynomials for secant. Taking the signs into account we obtain the
derivative polynomials for hyperbolic secant. For the Tchebyshev
transform of a Boolean 
algebra this was first observed in~\cite[Corollary 9.3]{Hetyei-mfp}, but
at the level of counting faces in the order complex of a graded poset
there is no difference between considering the operator $P\mapsto T(P)$
and the operator $P\mapsto \grI(P)$. It has been shown
in~\cite{Hetyei-tt} that the derivative 
polynomials for hyperbolic tangent and hyperbolic secant have interlaced
real roots in the interval $[-1,1]$.  As noted in the same paper, the
$F$-polynomial $F_\triangle(t)$ and the $h$-polynomial
$h_{\triangle}(t)$ of a $(d-1)$-dimensional simplicial complex
$\triangle$  are connected by the formula 
$$
(1-t)^d\cdot F_{\triangle}\left(\frac{1+t}{1-t}\right)=(1-t)^d\sum_{j=0}^d f_j
\left(\frac{t}{1-t}\right)^j=h_{\triangle}(t).$$
Hence the real-rootedness of the $h$-polynomial
of the type $B$ Coxeter complex is also a consequence of the fact that
the derivative polynomials for the hyperbolic secant have real roots.
}
\end{remark}

\section{Computing the flag $f$-vector of the graded poset of intervals}
\label{sec:flagf}

In this section we review how for any graded poset $P$, the flag
$f$-vector of its graded poset of intervals $\grI(P)$ may be obtained
from the flag $f$-vector of $P$ by a linear transformation. Such
formulas were first found by Joji\'{c}~\cite{Jojic}. At the end of the
section we will also present a more direct proof of his key formulas. By
``chain'' in this section we always mean a chain containing the unique minimum
element and the unique maximum element. This treatment is equivalent to
excluding both of these elements from all chains. 

\begin{definition}
Given a chain $\emptyset\subset [u_1,v_1]\subset [u_2,v_2]\subset \cdots
\subset[u_k,v_k]\subset [u_{k+1},v_{k+1}]=[\0,\1]$ in the graded poset
of intervals $\grI(P)$ of a graded poset $P$, we call the set
$$\{u_1,v_1,u_2,v_2,\ldots,u_{k+1},v_{k+1}\}$$
the {\em support} of the chain.
\end{definition}  

Obviously the support of a chain in $\grI(P)$ is a chain in $P$
containing the minimum element $\0$ and the maximum element $\1$.

The next statement expresses the number of chains in $\grI(P)$ having
the same support in terms of the {\em Pell numbers $P(n)$}. 
These numbers are given by the initial conditions $P(1)=1$
and $P(2)=2$ and by the recurrence $P(n)=2\cdot P(n-1)+P(n-2)$ for
$n\geq 3$. A detailed bibliography on the Pell numbers may be found at sequence
A000129 of~\cite{OEIS}.

\begin{proposition}
\label{prop:pell2}
Let $P$ be a graded poset and let $c: \0=z_0<z_1<\cdots<z_{m-1}<z_{m}=\1$ be
a chain in it. Then the number of chains
$\emptyset\subset [u_1,v_1]\subset [u_2,v_2]\subset \cdots
\subset[u_k,v_k]\subset [u_{k+1},v_{k+1}]=[\0,\1]$ whose support is $c$
is the sum $P(m)+P(m+1)$ of two adjacent Pell numbers. 
\end{proposition}  
\begin{proof}  
We proceed by induction on $m$. For $m=1$ there are three
chains: $\emptyset\subset [\0,\1]$, $\emptyset\subset [\0,\0]\subset
[\0,\1]$ and $\emptyset\subset [\1,\1]\subset [\0,\1]$. 
For $m=2$, there are the following seven chains with support
$\0<z_1<\1$: 
\begin{enumerate}
\item $\emptyset\subset [\0,z_1]\subset [\0,\1]$,
\item $\emptyset\subset [\0,\0]\subset [\0,z_1]\subset [\0,\1]$,
\item $\emptyset\subset [z_1,z_1]\subset [\0,z_1]\subset [\0,\1]$,
\item $\emptyset\subset [z_1,\1]\subset [\0,\1]$,
\item $\emptyset\subset [\1,\1]\subset [z_1,\1]\subset [\0,\1]$,
\item $\emptyset\subset [z_1,z_1]\subset [z_1,\1]\subset [\0,\1]$, and
\item $\emptyset\subset [z_1,z_1]\subset [\0,\1]$. 
\end{enumerate}
Let us list the elements of the chain in $\grI(P)$ in decreasing order.   
The largest element of the chain must be $[\0,\1]$, the
unique maximum element. The next element is either the interval
$[z_1,\1]$ or the interval $[\0,z_m]$ or the interval $[z_1,z_m]$. We
can not make the minimum of this next interval larger than $z_1$ because
that would force skipping $z_1$ in the support, similarly the maximum of
this next interval is at least $z_m$. Applying the induction hypothesis
to the intervals $[z_1,\1]$, $[\0,z_m]$ and $[z_1,z_m]$, respectively,
we obtain that the number of chains is 
$$2\cdot (P(m)+P(m+1))+(P(m-1)+P(m))=P(m+1)+P(m+2).$$
\end{proof}

\begin{remark}
{\em The numbers $P(n)+P(n+1)$ are listed as sequence
  A001333 in~\cite{OEIS}. They are known as the numerators of the
  continued fraction convergents to $\sqrt{2}$, and have many
  combinatorial interpretations. The even, respectively odd
  indexed entries in this sequence may also be obtained by substitutions
  into the Tchebyshev polynomials of the first, respectively second kind.}
\end{remark}

It is transparent in the proof of Proposition~\ref{prop:pell2}
that the contributions of chains of $\grI(P)$ with a fixed
support to $\Upsilon_{\grI(P)}(a,b)$ depends only on the contribution of
their support to $\Upsilon_{P}(a,b)$. This observation
motivates the following definition.
\begin{definition}
Given an $ab$-word $w$ of degree $n$, we define $\iab(w)$ as 
the contribution of all chains of $\grI(P)$ with a fixed
support to $\Upsilon_{\grI(P)}(a,b)$, whose support is the same chain
of $P$, contributing the word $w$ to $\Upsilon_{P}(a,b)$. 
\end{definition}  

\begin{theorem}
\label{thm:itransf}
  The operator $\iab$ may be recursively computed using the following formulas.
  \begin{enumerate}
  \item $\iab(a^n)=(a+2b)a^n$ holds for $n\geq 0$. In particular, for
    the empty word $\varepsilon$ we have $\iab(\varepsilon)=(a+2b)$.
  \item $\iab(a^iba^j)=(a+2b)(a^iba^j+a^jba^i)+ba^{i+j+1}$ holds for
    $i,j\geq 0$.
  \item
    $\iab(a^ibwba^j)=\iab(a^ibw)ba^j+\iab(wba^j)ba^i+\iab(w)ba^{i+j+1}$
    holds for $i,j\geq 0$ and any $ab$-word $w$.   
\end{enumerate}    
\end{theorem}  
\begin{proof}
The only chain that contributes $a^n$ to the $ab$-index of a graded
poset is the chain $\0<\1$ in a graded poset $P$ of rank $n+1$. As seen in
the proof of Proposition~\ref{prop:pell2}, there are $3$ chains in
$\grI(P)$ whose support is  $\0<\1$, and their contribution is to the
$ab$-index of $\grI(P)$ is $(a+2b)a^n$.

Similarly, the only chains that contribute $a^iba^j$ to the $ab$-index
of a graded poset are the chains $\0<z_1<\1$ in a graded poset $P$ of rank
$i+j+2$, where the rank of $z_1$ is $i+1$. As seen in 
the proof of Proposition~\ref{prop:pell2}, there are $7$ chains in
$\grI(P)$ whose support is  $\0<z_1<\1$, and their contribution is to the
$ab$-index of $\grI(P)$ is $(a+2b)(a^iba^j+a^jba^i)+ba^{i+j}$.

Finally, consider a chain $c:\0<z_1<z_2<\cdots <z_{k}<z_{k+1}=\1$ that
contributes $a^ibwba^j$ to the $ab$-index of a graded poset $P$ of rank
$n+1$. In such a chain the rank of $z_1$ is $i+1$ and the rank of of
$z_k$ is $n-j$.  The
largest element below $[\0,\1]$ of any chain in $\grI(P)$ with support
$c$ is either $[\0,z_k]$ (of rank $n-j+1$)  or $[z_1,\1]$ (of rank
$n+1-i$) or $[z_1,z_k]$ (of rank $n-i-j+1$). The three terms correspond
to the contributions of the chains of these three types. 
\end{proof}

\begin{corollary}
There is a linear map $I_n:{\mathbb R}^{2^n}\rightarrow {\mathbb
  R}^{2^{n+1}}$ sending the flag $f$-vector of each graded poset $P$ of
rank $n+1$  into the flag $f$-vector of its graded poset of intervals
$\grI(P)$. This linear map may be obtained by encoding flag $f$-vectors
with the corresponding upsilon invariants, and extending the map $\iab$
by linearity. 
\end{corollary}  

\begin{example}
\label{ex:iab}
{\em
Using Theorem~\ref{thm:itransf} we obtain the following formulas.
\begin{itemize}
\item[$n=1$:] $\iab(a)=a^2+2ba$, $\iab(b)=(a+2b)(b+b)+ba=4b^2+2ab+ba$. 
\item[$n=2$:] $\iab(a^2)=a^3+2ba^2$,
  $\iab(ab)=(a+2b)(ab+ba)+ba^2=a^2b+aba+2bab+2b^2a+ba^2$,
  $\iab(ba)=a^2b+aba+2bab+2b^2a+ba^2=\iab(ab)$, and 
\begin{align*}
  \iab(b^2)&=2\iab(b)b+\iab(\varepsilon)ba=2(4b^2+2ab+ba)b+(a+2b)ba\\
  &=8b^3+4ab^2+2bab+aba+2b^2a.
\end{align*}  
\end{itemize}
}
\end{example}

We conclude this section by a shorter proof of two key formulas
found by Joji\'{c}~\cite{Jojic}. These express the connection between
the $ab$-indices of $P$ and $\grI(P)$ using a coproduct operation first
introduced by Ehrenborg and Readdy~\cite{Ehrenborg-Readdy-cop}.
\begin{definition}
The coproduct $\Delta$ is defined on the algebra ${\mathbb Q}\langle a,
b\rangle$, whose basis is the set of all $ab$-words, as follows. Given
an $ab$-word $u=u_1u_2\cdots u_n$, we set
$$
\Delta(u)=\sum_{i=1}^n u_1\cdots u_{i-1}\otimes u_{i+1}\cdots u_n. 
$$
\end{definition}  
Note that the algebra  ${\mathbb Q}\langle a,b\rangle$ also includes the
empty word $1$ as a basis vector. In the next theorem we will use {\em
  Sweedler notation}
$$
\Delta(u)=\sum_{u} u_{(1)} \otimes u_{(2)}
$$
to refer to the coproduct and the notation $u^*$ as a
shorthand for $u^*=u_nu_{n-1}\cdots u_1$, obtained by reversing the word
$u=u_1u_2\cdots u_n$. 

\begin{theorem}[Joji\'{c}]
\label{thm:Jojicab}
Given a graded poset $P$ of rank $n+1$, we have
$\Psi_{\grI(P)}(a,b)=\Iab(\Psi_{P}(a,b))$, where the linear operator
$\Iab: {\mathbb Q}\langle a,b\rangle \rightarrow {\mathbb Q}\langle a,b\rangle$
is defined by the following recursive formulas on the basis of
$ab$-words:
\begin{align}
\Iab(u\cdot a)&=\Iab(u)\cdot a+(ab+ba)\cdot u^* +\sum_{u}
\Iab(u_{(2)})\cdot ab\cdot u_{(1)}^*\\
\Iab(u\cdot b)&=\Iab(u)\cdot b+(ab+ba)\cdot u^* +\sum_{u}
\Iab(u_{(2)})\cdot ba\cdot u_{(1)}^*.
\end{align}  
\end{theorem}
\begin{proof}
We prove the theorem by showing the following, linearly equivalent
formulas for the operator $\iab$:
\begin{align}
\label{eq:ua}
\iab(ua)&=\iab(u)a+\sideset{}{'}\sum_{u}\iab(u_{(2)})\cdot
(ab-ba)\cdot u_{(1)}^*\\
\label{eq:ub}
\iab(ub)&=\iab(u)b+(ab+ba+2b^2)\cdot
u^*+\sideset{}{'}\sum_{u}\iab(u_{(2)})\cdot b(a+b)\cdot u_{(1)}^*.
\end{align}  
Here the symbol $\sum_{u}^{\prime}$ is Sweedler notation for the coproduct
$\Delta'$ defined on the algebra ${\mathbb Q}\langle a,
b\rangle$ by the rule
$$
\Delta'(a^{i_1}ba^{i_2}b\cdots a^{i_k}ba^{i_{k+1}})
=\sum_{j=1}^k a^{i_1}ba^{i_2}b\cdots a^{i_j}\otimes
a^{i_{j+1}}\cdots a^{i_k}ba^{i_{k+1}}. 
$$
Equations~(\ref{eq:ua}) and (\ref{eq:ub}) are linearly equivalent to the
formulas stated in the theorem because
$ab$-index $\Psi_P(a,b)$ is obtained by substituting $e=a-b$ into
$\Upsilon_P(a,b)$, and the definition of $\Delta'$ corresponds to the
rule
$$
\Delta(e^{i_1}be^{i_2}b\cdots e^{i_k}be^{i_{k+1}})
=\sum_{j=1}^k e^{i_1}be^{i_2}b\cdots e^{i_j}\otimes
e^{i_{j+1}}\cdots e^{i_k}be^{i_{k+1}}
$$
whose verification is left to the reader. After substituting $a-b$ into
$a$,  equations~(\ref{eq:ua}) and (\ref{eq:ub}) become
\begin{align*}
\Iab(u(a-b))&=\Iab(u)(a-b)+\sum_{u}\Iab(u_{(2)})\cdot
(ab-ba)\cdot u_{(1)}^*\\
\Iab(ub)&=\Iab(u)b+(ab+ba)\cdot
u^*+\sum_{u}\Iab(u_{(2)})\cdot ba\cdot u_{(1)}^*.
\end{align*}  
The sum of these equations is the first equation stated in the theorem,
while the second equation is the same in both pairs.

To prove (\ref{eq:ub}), note that $\iab(ub)$ is the sum of all
$ab$-words associated to chains of intervals supported by a fixed chain
$\0<x_1<\cdots <x_k<x_{k+1}<\1$ where the rank of $x_{k+1}$ is one less than
the rank of $\1$ and the set of ranks of $x_1<\cdots <x_k$ is marked by
the $ab$-word $u$. The summand $\iab(u)b$ is contributed by all chains
of intervals containing $[\0,x_{k+1}]$. These chains can not contain any
interval containing $\1$, except for $[\0,\1]$, so the remaining
intervals in all such chains are contained in $[\0,x_{k+1}]$.
The sum $\sum_{u}'\iab(u_{(2)})\cdot b(a+b)\cdot u_{(1)}^*$ is
contributed by all chains of intervals, containing some intervals of the
form $[x_j,x_{k+1}]$, for some $j\geq 1$ but not containing
$[\0,x_{k+1}]$. Let $i\geq 1$ be the least index 
for which $[x_i,x_{k+1}]$ belongs to the chain. the factor
$\iab(u_{(2)})b$ is contributed by the intervals contained in this
interval and by the interval $[x_i,x_{k+1}]$ itself. For each $j<i$ the
interval $[x_j,\1]$ must belong to the chain, these contribute the
factor $au_{(1)}^*$. The factor $(a+b)$ right after $\iab(u_{(2)})b$
reflects the possibility of adding to the chain $[x_j,\1]$ or omitting
it. This choice may be done independently of everything else. 
The remaining terms are contributed by chains not
containing any interval of the form $[\0,x_{k+1}]$ or $[x_i,x_{k+1}]$.
We claim that all these remaining chain contribute the term
$(ab+ba+2b^2)\cdot u^*$. Since $x_{k+1}$ must be part of the support,
these chains contain at least one of $[x_{k+1},\1]$ and $[x_{k+1},x_{k+1}]$.
The intersection of such a chain of intervals with the set $\{[\1,\1],
[x_{k+1},\1], [x_{k+1},x_{k+1}] \}$ can be $\{[x_{k+1},\1],\}$,
$\{[x_{k+1},x_{k+1}]\}$, $\{[\1,\1], [x_{k+1},\1],\}$, or 
$\{[x_{k+1},\1], [x_{k+1},x_{k+1}] \}$. These four possibilities account
for the presence of a factor $(ab+ba+2b^2)$. For all $i\leq k$ the
interval $[x_i,\1]$ must be present, this explains the presence of the
factor $u^*$.  

To prove (\ref{eq:ua}), note that $\iab(ua)$ is the sum of all
$ab$-words associated to chains of intervals supported by a fixed chain
$C:\0<x_1<\cdots <x_k<\1$ where the difference between the rank of $x_{k}$
and the rank of $\1$ is at least $2$ and the set of ranks of $x_1<\cdots
<x_k$ is marked by the $ab$-word $u$. Let us fix a coatom $x_{k+1}$ in
the interval $[x_k,\1]$. There is an obvious bijection between the
chains of intervals supported $C$ and the chains of intervals
supported by the chain $C':\0<x_1<\cdots <x_k<x_{k+1}$. This bijection
is induced by replacing each occurrence of $\1$ by $x_{k+1}$. Clearly
the sum of the $ab$-words of all intervals supported by $C'$ is
$\iab(u)$, let us multiply this sum by $a$ on the right and compare the
contribution of a chain of intervals supported by $C$ to $\iab(ua)$ with
the contribution of the corresponding chain of intervals supported by
$C'$ to $\iab(u)a$. 

{\bf\noindent Case 1:} The chain of intervals supported by $C$ contains
no interval of the form $[x_i,\1]$. The contribution of such a chain of
intervals to $\iab(ua)$ is the same as the contribution of the
corresponding chain of intervals supported by $C'$ to $\iab(u)a$.

{\bf\noindent Case 2:} The chain of intervals supported by $C$ contains
an interval of the form $[x_i,\1]$. Let us take the largest $i$ with this
property, that is, the least interval of this form, and let us
break the $ab$ word corresponding to chain of intervals at the letter
$b$ corresponding to this interval. The intervals of the
chain of intervals contained in $[x_i,\1]$ do not contain $\1$ and they
are the same in the corresponding chain of intervals supported by
$C'$. Their contribution is $\iab(u_{(2)})$. The remaining intervals
of the original chain of intervals supported by $C$ are all intervals of
the form $[x_j,\1]$ where $j<i$. In the corresponding chain of intervals
each $[x_j,\1]$ needs to be replaced by $[x_j,x_{k+1}]$, the rank of the
corresponding interval is one less: the contribution of such a chain of
intervals is thus $ab u_{(1)}^*$ to $\iab(ua)$, and the contribution of
the corresponding intervals to $\iab(u)a$ is thus $ba u_{(1)}^*$.  
\end{proof}

\section{Interval transforms of the second kind}
\label{sec:t2}

In Section~\ref{sec:itcheb} we have seen that for any poset $P$, the
order complex of the poset of intervals $I(P)$ is a Tchebyshev
triangulation of the order complex of $P$. In this setting, the elements
of the original poset $P$ are identified with the singleton intervals in
$I(P)$. Hence we make the following definition.

\begin{definition}
Given a partially ordered set $P$ we define its {\em interval transform
  of the second kind $\II(P)$} the multiset of subposets of $I(P)$
defined as follows: for each $x\in P$ we take the subposets of $I(P)$
formed by all elements $[y,z]\in I(P)$ containing $[x,x]$. 
\end{definition}  

It is a direct consequence of the definition and Theorem~\ref{thm:it} we
obtain the following.

\begin{corollary}
\label{cor:iit}
For any finite poset $P$ the order complex of $\II(P)$ is a Tchebyshev
triangulation of the second kind of the order complex of $P$, associated
to the Tchebyshev triangulation of the first kind that is the order
complex of $I(P)$. 
\end{corollary}

\begin{corollary}
\label{cor:iit2}  
If $P$ is a graded poset then its interval transform of the second kind
is the collection of all intervals of the form
$[[x,x], [\0,\1]]\subset I(P)$ for each $x\in P$.    
\end{corollary}

\begin{definition}
\label{def:grII}  
Let $(P_1,\ldots,P_m)$ be a list of graded posets. We extend the
notions of the upsilon invariant and $ab$-index by linearity, i.e., we
set
$$
\Upsilon_{(P_1,\ldots,P_m)}(a,b)=\sum_{i=1}^m \Upsilon_{P_i}(a,b)
\quad\mbox{and}\quad
\Psi_{(P_1,\ldots,P_m)}(a,b)=\sum_{i=1}^m \Psi_{P_i}(a,b)
$$
For a graded poset $P$ we then define the {\em total $ab$-index}
$\Psi_{\grII(P)}(a,b)$ of $\grII(P)$ by 
$$\Psi_{\grII(P)}(a,b)=\sum_{u\in P} \Psi_{[[x,x], [\0,\1]]}(a,b).$$
\end{definition}  

The following statement is straightforward.
\begin{proposition}
\label{prop:i2prod}  
Let $P$ be a graded poset. For each $x\in P$, the set of intervals
$[y,z]$ contained in $[[x,x],[\0,\1]]\subset \grI(P)$ and
ordered by inclusion is isomorphic to the direct product $[\0,x]^*\times
[x,\1]$. Here $[\0,x]^*$ is the {\em dual} of the poset $[\0,x]$,
obtained by reversing the order of $[\0,x]$. 
\end{proposition}

Proposition~\ref{prop:i2prod} allows us to compute the effect of taking
the interval transform on the second kind on the $ab$-index of a poset
using the {\em mixing operator} introduced by Ehrenborg and
Readdy~\cite[Definition 9.1]{Ehrenborg-Readdy-cop}   
\begin{definition}
The {\em mixing operator $M$} is a bilinear operator defined on the
noncommutative algebra ${\mathbb Q}\langle a, b\rangle$ as the
follows. For each pair of $ab$-words $(u,v)$ we set 
$$
M(u,v)=\sum_{r=1}^2\sum_{s=1}^2 \sum_{n-r-s-1 \mbox{ is even}} M_{r,s,n}(u,v). 
$$
Here the operators $M_{r,s,n}(u,v)$ are recursively defined by
\begin{align*}
  M_{1,2,2}(u,v)&= u\cdot a\cdot v,\\
  M_{2,1,2}(u,v)&= u\cdot b\cdot v,\\
  M_{1,s,n+1}(u,v)&=\sum_{u} u_{(1)}\cdot a\cdot
  M_{2,s,n}(u_{(2)},v)\quad \mbox{and}\\
  M_{2,s,n+1}(u,v)&=\sum_{v} v_{(1)}\cdot b\cdot M_{1,s,n}(u,v_{(2)}).\\
\end{align*}  
\end{definition}   

It has been shown by Ehrenborg and
Readdy~\cite[Theorem 9.2]{Ehrenborg-Readdy-cop} that the $ab$-index of
the direct product of the graded posets $P$ and $Q$ is given by 
\begin{align}
  \Psi_{P\times Q}(a,b)=M(\Psi_P(a,b),\Psi_Q(a,b)).
\label{eq:mixing}  
\end{align}

Combining Equation~(\ref{eq:mixing}) with Corollary~\ref{cor:iit2} we may
compute the total $ab$-index of an interval transform of a
second kind as follows. 
\begin{theorem}
\label{thm:Jojicab2}  
Given a graded poset $P$ of rank $n+1$, we have
$\Psi_{\grII(P)}(a,b)=\IIab(\Psi_{P}(a,b))$, where the linear operator
$\IIab: {\mathbb Q}\langle a,b\rangle \rightarrow {\mathbb Q}\langle a,b\rangle$
is given by the formula
$$
\IIab(u)=u+u^*+\sum_u M(u_{(1)}^*,u_{(2)}).
$$
\end{theorem}  
\begin{proof}
By Definition~\ref{def:grII} and Proposition~\ref{prop:i2prod} we have
$$
\Psi_{\grII(P)}(a,b)=\sum_{x\in P} \Psi_{[\0,x]^*\times [x,\1]}(a,b)
$$

By Equation~(\ref{eq:mixing}) this may be rewritten as
\begin{align*}
\Psi_{\grII(P)}(a,b)&=
\sum_{x\in P} M(\Psi_{[\0,x]^*}(a,b),\Psi_{[x,\1]}(a,b))\\
&=
\Psi_{[\0,\1]}(a,b)^*+\Psi_{[\0,\1]}(a,b)+\sum_{\0<x<\1}
M(\Psi_{[\0,x]}(a,b)^*,\Psi_{[x,\1]}(a,b)). 
\end{align*}
The statement is now a direct consequence of~\cite[Equation
  (3.1)]{Ehrenborg-Readdy-cop}, stating
$$
\Delta{\Psi_{P}}(a,b)=\sum_{\0<x<\1} \Psi_{[\0,x]}(a,b)\otimes
  \Psi_{[x,\1]}(a,b).  
$$
\end{proof}  

In analogy to~\cite[Theorem 10.10]{Ehrenborg-Readdy-Tch} we may find
many eigenvalues and eigenvectors of the operator $\II: {\mathbb Q}\langle
a,b\rangle \rightarrow {\mathbb Q}\langle a,b\rangle$. The quest to find
the eigenvalues of $\II$ is complicated by the fact that this linear
operator has a nontrivial kernel. To find part of this kernel, we first extend
the operator $u\mapsto u^*$ by linearity to all $ab$-polynomials.  
\begin{corollary}
\label{cor:kernel}
If the homogeneous $ab$-polynomial $u\in {\mathbb Q}\langle
a,b\rangle_n$ satisfies $u^*=-u$ then $\II(u)=0$
\end{corollary}
Corollary~\ref{cor:kernel} is a direct consequence of
Theorem~\ref{thm:Jojicab2}. It inspires decomposing the vectorspace
${\mathbb Q}\langle a,b\rangle_n$ of $ab$-polynomials of degree $n$ into
a direct sum of the vector spaces of {\em symmetric} and {\em
  antisymmetric} $ab$-polynomials.
\begin{definition}
A homogeneous $ab$-polynomial $u\in{\mathbb Q}\langle a,b\rangle_n$ of
degree $n$ is {\em symmetric} if it satisfies $u^*=u$ and {\em
  antisymmetric} if it satisfies $u^*=-u$. We denote the vectorspace of
symmetric, respectively antisymmetric $ab$-polynomials of degree $n$ by
$\Sym{n}$, respectively $\Asym{n}$.  
\end{definition}  
\begin{proposition}
  The vectorspace ${\mathbb Q}\langle a,b\rangle_n$ may be written as
  the direct sum
  $$
{\mathbb Q}\langle a,b\rangle_n=\Sym{n}\oplus \Asym{n}.
$$
Here $\dim \Asym{n}=2^{n-1}-2^{\lfloor(n-1)/2\rfloor}$ and $\dim
\Sym{n}=2^{n-1}+2^{\lfloor(n-1)/2\rfloor}$.  
\end{proposition}
\begin{proof}
Clearly $\Sym{n}\cap \Asym{n}=0$ and each $u\in {\mathbb Q}\langle a,b\rangle_n$
may be written as
$$
u=\frac{1}{2}\cdot (u+u^*) +\frac{1}{2}\cdot (u-u^*),  
$$
where $u+u^*\in\Sym{n}$ and $u-u^*\in\Asym{n}$. The dimension
formulas are direct consequences of the fact that the number of {\em
  symmetric}  $ab$-words $w$ satisfying $w=w^*$ is $2^{\lfloor
  (n+1)/2\rfloor}$ and hence the number of {\em
  asymmetric}  $ab$-words $w$ satisfying $w\neq w^*$ is $2^n-2^{\lfloor
  (n+1)/2\rfloor}$. Asymmetric $ab$-words $w$ form pairs $\{w,w^*\}$,
and we may associate to each such unordered pair a vector $w-w^*$, these
vectors form a basis of $\Asym{n}$.
\end{proof}  

Next we show the following analogue of~\cite[Proposition
  10.9]{Ehrenborg-Readdy-Tch}.  
\begin{lemma}
\label{lemma:Ulift}  
If the homogeneous $ab$-polynomial $u\in {\mathbb Q}\langle
a,b\rangle_n$ of degree $n$ is an eigenvector of the linear operator
$\II: {\mathbb Q}\langle a,b\rangle \rightarrow {\mathbb Q}\langle
a,b\rangle$ then so is the homogeneous $ab$-polynomial
$\lift(u):=(a-b)u+u(a-b)\in 
{\mathbb Q}\langle a,b\rangle_{n+1}$. Both eigenvectors have the same
eigenvalue.  
\end{lemma}
\begin{proof}
  Assume $\II(u)=\lambda\cdot u$ holds. By
  Theorem~\ref{thm:Jojicab2} we may write 
\begin{align*}
  \II((a-b)u+u(a-b))&=(a-b)u+u^*(a-b)+u(a-b)+(a-b)u^*\\
  &+(a-b)\sum_u M(u_{(1)}^*,u_{(2)})+\sum_u M(u_{(1)}^*,u_{(2)})(a-b)\\
  &=(a-b)\II(u)+\II(u)(a-b)=\lambda\cdot ((a-b)u+u(a-b)). 
\end{align*}    
\end{proof}
Note that the restriction of the operator ${\mathcal L}$ to $\Sym{n}$
takes $\Sym{n}$ into $\Sym{n+1}$. The analogue of~\cite[Proposition
  10.8]{Ehrenborg-Readdy-Tch} may be stated in more general terms, as follows.
\begin{proposition}
  \label{prop:Uprod}
  For any pair of graded partially ordered sets $P$ and $Q$,
  $$
\II(P\times Q)=\II(P) (\times) \II(Q) \quad\mbox{holds.}
$$
Here $\II(P) (\times) \II(Q)$ denotes the multiset of posets
$\{P_1\times Q_1\::\: P_1\in \II(P), Q_1\in \II(Q)\}$. 
\end{proposition}
\begin{proof}
The set $\II(P\times Q)$ is the multiset of all intervals of $P\times
Q$, of the form
$$\left[\left[(p,q),(p,q)\right], \left[(\0_P,\0_Q),(\1_P,\1_Q)\right]\right]$$
ordered by inclusion. Here $p$ ranges over all elements of $P$ and $Q$
independently ranges over all elements of $Q$.
The statement follows from the obvious isomorphism 
$$\left[\left[(p,q),(p,q)\right], \left[(\0_P,\0_Q),(\1_P,\1_Q)\right]\right]\cong
\left[\left[p,p\right],\left[\0_P,\1_P\right]\right]\times \left[\left[q,q\right],\left[\0_Q,\1_Q\right]\right].
$$
\end{proof}  

\begin{corollary}
\label{cor:Uprod}  
Assume the homogeneous $ab$-polynomial $u_i$ is an
eigenvector with eigenvalue $\lambda_i$ of the interval transform of the
second kind $\II$ for $i = 1, 2$. Then $M(u_1, u_2)$ is an eigenvector with 
eigenvalue $\lambda_1\cdot \lambda_2$. 
\end{corollary}  
As a consequence of~\cite[Corollary 4.3]{Ehrenborg-Fox}, if $u_i\in \Sym{n_i}$
holds for $i=1,2$ then  $M(u_1,u_2)\in \Sym{n_1+n_2}$.

In~\cite{Ehrenborg-Readdy-Tch} find all eigenvalues and eigenvectors of
the Tchebyshev operator of the second kind, by repeated use of the
analogues of Lemma~\ref{lemma:Ulift} and
Proposition~\ref{prop:Uprod}. More precisely, \cite[Theorem
  10.10]{Ehrenborg-Readdy-Tch} states that a basis of eigenvectors may
be generated by repeatedly using the pyramid operator
$\operatorname{Pyr}: u\mapsto M(1,u)$ and their variant of the lifting
operator $\lift$ which sends $u$ into $(a-b)u$. A key ingredient of
their proof is the use of the fact that the intersection of the ranges
of the pyramid operator and their lifting operator is zero. In our case
this is not true any more, furthermore our operator $\II$ has a
nontrivial kernel. That said, we make the following conjectures.
\begin{conjecture}
For each $n\geq 1$, the kernel of $\II: {\mathbb Q}\langle
a,b\rangle_n\rightarrow {\mathbb Q}\langle a,b\rangle_{n+1}$ is $\Asym{n}$.
\end{conjecture}  

\begin{conjecture}
For each $n\geq 1$ a generating set of $\Sym{n}$, consisting of
eigenvectors only may be found by taking all possible $n$-fold compositions of
the pyramid operator $\operatorname{Pyr}$ and of the lift operator
$\lift$, and applying it to $1$.  
\end{conjecture}

\section{The graded poset of intervals of an Eulerian poset}
\label{sec:Eulerian}

The following result is due to C.\ Athanasiadis~\cite[Proposition
  2.5]{Athanasiadis}, it is also stated in a special case by
Joji\'{c}~\cite[Remark 10]{Jojic}.   
\begin{proposition}
  \label{prop:Eulerian}
If a graded poset $P$ is Eulerian then the same holds for the graded
poset of its intervals $\grI(P)$.
\end{proposition}  
Indeed, it is well known consequence of Phillip Hall's theorem
(see~\cite[Propositition 3.8.5]{Stanley-EC1}) that a graded poset is
Eulerian if and only if the reduced characteristic of the order complex
of each open interval $(u,v)$ is $(-1)^{\rho(v)-\rho(u)}$ where $\rho$
is the rank function. Since taking the graded poset of intervals results
in taking a triangulation of the suspension of each such order complex,
the reduced Euler characteristic remains unchanged.

As a consequence of Proposition~\ref{prop:Eulerian}, the linear map $I_n$
takes the flag $f$-vector of any graded Eulerian poset of rank $n+1$
into the flag $f$-vector of a graded Eulerian poset of rank $n+2$. As a
direct consequence of Theorem~\ref{thm:Jojicab} the following formulas 
hold, see~\cite[Corollary 7]{Jojic}:
\begin{corollary}[Joji\'{c}]
\label{cor:Jojiccd}
Given a graded poset $P$ of rank $n+1$, we have
$\Psi_{\grI(P)}(c,d)=\Iab(\Psi_{P}(c,d))$, where the linear operator
$\Iab: {\mathbb Q}\langle c,d\rangle \rightarrow {\mathbb Q}\langle c,d\rangle$
is defined by the following recursive formulas on the basis of
$ab$-words:
\begin{align*}
\Iab(u\cdot c)&=\Iab(u)\cdot c+2d\cdot u^* +\sum_{u}
\Iab(u_{(2)})\cdot c\cdot u_{(1)}^*\\
\Iab(u\cdot d)&=\Iab(u)\cdot d+(dc+cd)\cdot u^*+d\cdot u^*\cdot c\\
&+\sum_{u}
(\Iab(u_{(2)})\cdot d\cdot \operatorname{Pyr}(u_{(1)}^*)+d\cdot
u_{(2)}^*\cdot d\cdot u_{(1)}^*).
\end{align*}  
Here $\operatorname{Pyr}$ is the the linear operator defined by
Ehrenborg and Readdy~\cite{Ehrenborg-Readdy-cop} associating to the $cd$
index of each poset $P$ the $cd$-index of $P\times B_1$ where $B_1$ is
the Boolean algebra of rank $1$. 
\end{corollary}

Proposition~\ref{prop:Eulerian} has the following easy consequence.
\begin{corollary}
The interval transform of the second kind $\grII(P)$ of any graded
Eulerian poset $P$ of rank $n+1$ is a multiset of Eulerian posets of
rank $n+1$. Hence $\Psi_{\grII(P)}(a,b)$ is a polynomial of $c$ and $d$  
\end{corollary}  
We will use the notation $\Psi_{\grII(P)}(c,d)$ to stand for the
polynomial $\Psi_{\grII(P)}(a,b)$ rewritten as an expression of $c$ and
$d$. In analogy to Theorem~\ref{thm:Jojicab} and Corollary~\ref{cor:Jojiccd},
the restriction of the map $\IIab$ to $cd$-polynomials allows us to
compute $\Psi_{\grII(P)}(c,d)$. In doing so, the following formulas of
Ehrenborg and Fox~\cite[Theorem 5.1]{Ehrenborg-Fox} are helpful. For two
$cd$ monomials $u$ and $w$ we have
\begin{align}
\label{eq:vc}  
  M(u,v\cdot c)&=v\cdot d\cdot u+M(u,v)\cdot c+\sum_{u} M(u_{(1)},v)\cdot d
  \cdot u_{(2)}\quad\mbox{and}\\ 
\label{eq:vd}  
  M(u,v\cdot d)
  &=v\cdot d\cdot \operatorname{Pyr}(u)
  +M(u,v)\cdot d
  +\sum_{u} M(u_{(1)},v)\cdot d \cdot \operatorname{Pyr}(u_{(2)}).
\end{align}  
\begin{remark}
{\em To avoid introduction a second coproduct denoted $\Delta^*$ and the
counit $\epsilon$ used to state~\cite[Theorem 5.1]{Ehrenborg-Fox}, we
rewrote formulas (\ref{eq:vc})  and (\ref{eq:vd}) using
$M(\epsilon,v)=v$, $d\epsilon=c$ and $\operatorname{Pyr}(\epsilon)=1$.
}
\end{remark}  
Equations (\ref{eq:vc})  and (\ref{eq:vd}), together with the obvious
\begin{equation}
\label{eq:Mswap}  
M(u,v)=M(v,u),
\end{equation}
the initial condition
\begin{equation}
M(1,1)=c
\end{equation}
and the obvious
\begin{equation}
\operatorname{Pyr}(u)=M(1,u)
\end{equation}
allow to compute the function $M(u,v)$ in a recursive fashion.

\section{Special cases}
\label{sec:special}

In this section we compute the $cd$-indices of the poset of intervals
and of the interval transform of the second kind of two special posets:  
the ``ladder'' poset $L_n$ and the Boolean algebra $P([1,n])$ of rank $n$. 

The poset $L_n$ has exactly $2$ elements: $-i$ and $i$ for
each rank $i$ satisfying $0<i<n+1$, and any pair of elements at
different ranks are comparable. The poset $L_2$ of rank $3$ is
shown in Figure~\ref{fig:ladder}. To simplify our notation in the proof
of Theorem~\ref{thm:laddercd} below, we write the unique minimum element
of $L_n$ as $0$ and the unique maximum element as $n+1$. It is well,
known that $\Psi_{L_n}(c,d)=c^n$.

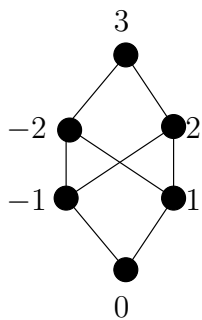
\begin{figure}[h]
  \input{butterfly_pdf.tex}
\caption{The ``ladder'' poset $L_2$ of rank $3$}  
\label{fig:ladder}
\end{figure}
The following result is due to Joji\'c~\cite[Theorem 9]{Jojic}
\begin{theorem}[Joji\'{c}]
\label{thm:laddercd}  
Assume that the finite vector  $(k_0,\ldots,k_{r})$ of nonnegative
integers satisfies $2r+k_0+k_2+\cdots+ k_{r}=n$. Then  
the coefficient of $c^{k_0}dc^{k_1} d\cdots c^{k_r}dc^{k_{r}}$ in
$\Iab(c^n)$ is $2^r(k_1+1)(k_2+1)\cdots (k_r+1)$.
\end{theorem}
\begin{remark}
  {\em As noted by Joji\'c, this formula is the dual of the one obtained for the
    other Tchebyshev transform see~\cite[Theorem 7.1]{Hetyei-tch}
    and~\cite[Corollary 6.6]{Ehrenborg-Readdy-Tch} (see also~\cite[Table
      1]{Hetyei-tch}, although
    the two poset operations are very 
    different. This observation also suggests that, when we comparing it 
    to the Tchebyshev transform, one would want to
    consider the dual of the poset of intervals, ordered by reverse
    inclusion.}    
\end{remark}  

Next we compute $\II(c^n)$. To do so, the following consequence of
Equation~(\ref{eq:vc}) will be useful:
\begin{align}
\label{eq:cij+1}
  M(c^i,c^{j+1})&=c^j\cdot d\cdot c^i+M(c^i,c^j)\cdot c
  +2\sum_{k=1}^i M(c^{k-1},c^{j})\cdot d\cdot c^{i-k}.
\end{align}  

\begin{lemma}
\label{lem:Mrec}  
The expressions $M(c^i,c^j)$ satisfy the recurrence
$$
M(c^{i+1},c^{j+1})=(M(c^{i},c^{j+1})+M(c^{i+1},c^{j}))c+M(c^i,c^j)\cdot
(2d-c^2)
$$
for $i,j\geq 0$
\end{lemma}  
\begin{proof}
Replacing $i$ with $i+1$ in (\ref{eq:cij+1}) yields
\begin{align}
\label{eq:ci+1j+1}  
  M(c^{i+1},c^{j+1})&=c^j\cdot d\cdot c^{i+1}+M(c^{i+1},c^j)\cdot c
  +2\sum_{k=1}^i M(c^{k-1},c^{j})\cdot d\cdot
  c^{i+1-k}\\
  &+2M(c^{i},c^{j})\cdot d. \nonumber
\end{align}
By multiplying both sides of (\ref{eq:cij+1}) by $c$ on the right we
obtain 
\begin{align}
\label{eq:cij+1c}
  M(c^i,c^{j+1})c&=c^j\cdot d\cdot c^{i+1}+M(c^i,c^j)\cdot c^2
  +2\sum_{k=1}^i M(c^{k-1},c^{j})\cdot d\cdot c^{i+1-k}.
\end{align}  
The statement now follows after subtracting (\ref{eq:cij+1c}) from
(\ref{eq:ci+1j+1}).
\end{proof}  
Using Lemma~\ref{lem:Mrec} it is easy to show the following
statement. Recall that a {\em Delannoy path} is a lattice path
consisting of East steps $(1,0)$, North steps $(0,1)$ and Northeast
steps $(1,1)$.
\begin{theorem}
\label{thm:mcc}  
$M(c^i,c^j)$ is the half of the total weight of all Delannoy paths from
  $(-1,0)$ or $(0,-1)$ to $(i,j)$ where each East step and North step
  has weight $c$ and each Northeast step has weight $(2d-c^2)$. The
  weight of each 
  Delannoy path is obtained by multiplying the weight of its steps, left
  to right, in the order from $(-1,0)$ or $(0,-1)$ to $(i,j)$.   
\end{theorem}  
\begin{proof}
  Let us define the function $\widetilde{M}(i,j)$ as follows:
  $$
  \widetilde{M}(i,j)=
  \begin{cases}
    \frac{1}{2} c^j & \mbox{if $i=-1$ and $j\geq 0$}\\
    \frac{1}{2} c^i & \mbox{if $j=-1$ and $i\geq 0$}\\
M(c^i,c^j) & \mbox{if $i,j\geq 0$}\\
  \end{cases}  
  $$
Note that the function $\widetilde{M}(i,j)$ is defined for all pairs of
integers $(i,j)$ satisfying $i,j\geq -1$, except for $i=j=-1$. It
suffices to show that the value of $\widetilde{M}(i,j)$ may be computed
as half of the total weight of the Delannoy paths stated above.

This statement is certainly true for $\widetilde{M}(i,-1)$ for $i\geq
0$: there is no Delannoy path from $(-1,0)$ to $(i,-1)$ and the only
Delannoy path from $(0,-1)$ to $(i,-1)$ is the lattice path consisting
of $i$ East steps. Similarly, there is only one Delannoy path form
$(-1,0)$ to $(-1,j)$, consisting of $j$ North steps and the statement
holds for $\widetilde{M}(-1,j)$.

Observe next that there are exactly two Delannoy paths from $(-1,0)$ or
$(0,-1)$ to $(0,0)$: the first consists of a single East step the second
consists of a single North step, their total weight is $M(1,1)=c$, as expected.

Next we show the validity of the statement for
$\widetilde{M}(i,0)$ when $i>0$. Note that the last step of any Delannoy path
ending at $(i,0)$ is either an East step from $(i-1,0)$  or a North step
from $(i,-1)$ or a Northeast step from $(i-1,-1)$. We want to show that
$$
\widetilde{M}(i,0)
=\widetilde{M}(i-1,0)\cdot c+\widetilde{M}(i,-1)\cdot c
+\widetilde{M}(i-1,-1)\cdot (2d-c^2),\quad\mbox{that
  is,}  
$$
$$
M(c^i,1)=M(c^{i-1},1)\cdot c+\frac{1}{2}\cdot c^{i+1}
+c^{i-1}\cdot d-\frac{1}{2}\cdot c^{i+1},
$$
which is equivalent to
$$
M(c^i,1)=M(c^{i-1},1)\cdot c+c^{i-1}\cdot d.
$$
This last equation is a direct consequence of (\ref{eq:Mswap}) and
(\ref{eq:vc}).  The proof of the statement for $\widetilde{M}(0,j)$ when
$j>0$ is completely analogous.

It remains to show the statement when both $i$ and $j$ are positive. For
these the last step of every Delannoy path ending at $(i,j)$ is either
an East step from $(i-1,j)$ or a North step from $(i,j-1)$ or a
Northeast step from $(i-1,j-1)$. The statement is a direct consequence
of Lemma~\ref{lem:Mrec}.   
\end{proof}
Recall that Stanley~\cite{Stanley-flag} introduced $e=a-b$ and noted
that the existence of the $cd$ index of an Eulerian poset is equivalent to
stating that the $ab$-index is a polynomial of $c$ and $e^2=c^2-2d$. In
terms of the resulting {\em $ce$-index}, Theorem~\ref{thm:mcc}  may be
restated as follows.
\begin{theorem}
\label{thm:mcce}  
The coefficient of $c^{k_0}e^2c^{k_1}e^2\cdots e^2c^{k_r}$ in
$M(c^i,c^j)$ is
$$
\frac{(-1)^r}{2}\cdot \binom{i+j+2-2r}{i+1-r}
$$
if $k_0+k_1+\cdots+k_r+2r=i+j+1$ and $0$ otherwise.  
\end{theorem}
\begin{proof}
By Theorem~\ref{thm:mcc}, a lattice path from $(-1,0)$ or $(0,-1)$ to
$(i,j)$ contributes a term $(-1)^r/2\cdot c^{k_0}e^2c^{k_1}e^2\cdots
e^2c^{k_r}$ exactly when there are $r$ Northeast steps, there are $k_0$
North or East steps before the first Northeast step, $k_r$ North or East
steps after the last Northeast step and there are exactly $k_i$
Northeast steps between the $i$th and $(i+1)$st Northeast step for
$i=1,\ldots,r-1$. Hence each term contributed must satisfy
$k_0+k_1+\cdots+k_r+2r=i+j+1$. 

Let us count first the number of lattice paths from $(-1,0)$ to $(i,j$)
contributing a term $(-1)^r/2\cdot 
c^{k_0}e^2c^{k_1}e^2\cdots e^2c^{k_r}$. The parameters $i$ and $j$ must
satisfy $i+1\geq r$ and $j\geq r$, as each Northeast step increases both
coordinates by $1$. Out of the $i+j+1-2r$ North or East steps we must
select $i+1-r$ East steps and $j-r$ North steps. This may be performed
$\binom{i+j+1-2r}{i+1-r}$ ways.

Similarly, the number of lattice paths from $(0,-1)$ to $(i,j$)
contributing a term $(-1)^r/2\cdot 
c^{k_0}e^2c^{k_1}e^2\cdots e^2c^{k_r}$ is
$\binom{i+j+1-2r}{j+1-r}=\binom{i+j+1-2r}{i-r}$. The stated result
follows by Pascal's identity.
\end{proof}  

\begin{remark}
{\em It is worth pointing out that the coefficient of
  $c^{k_0}e^2c^{k_1}e^2\cdots e^2c^{k_r}$ in $M(c^i,c^j)$ depends only
  on $i,j$ and $r$. For a fixed expression  $M(c^i,c^j)$ the
  coefficient of a $ce$-word depends only on the number of factors $e^2$
  in it.} 
\end{remark}

\begin{remark}
{\em For the somewhat similar diamond product, N.B.\ Fox gave a more general
lattice path interpretation~\cite[Theorem 5.4]{Fox}. It would be
interesting to see whether a similar approach could also help express
$M(u,v)$ in general as a total weight of lattice paths.} 
\end{remark}   

Using Theorem~\ref{thm:mcce} we may express $\II(c^n)$ as follows.
\begin{proposition}
\label{prop:uce}
Assume that the finite vector  $(k_0,\ldots,k_{r})$ of nonnegative
integers satisfies $2r+k_0+k_2+\cdots+ k_{r}=n-1$. Then  
the coefficient of $c^{k_0}e^2c^{k_1} e^2\cdots c^{k_{r-1}}e^2c^{k_{r}}$ in
$\II(c^n)$, written as a $ce$-polynomial, is $(-1)^r\cdot 2^{n+1-2r}$.
\end{proposition}  
\begin{proof}
Observe first that by the definition of the coproduct, the relation
$\Delta(c)=2\cdot 1\otimes 1$  and by
Theorem~\ref{thm:Jojicab2} we have 
$$
\IIab(c^n)=2\cdot c^n+2\sum_{i=0}^{n-1} M(c^{i},c^{n-1-i}).
$$
For positive $r$, by Theorem~\ref{thm:mcce} we get that the
coefficient of $c^{k_0}e^2c^{k_1} e^2\cdots c^{k_{r-1}}e^2c^{k_{r}}$ in
$\II(c^n)$ is
$$
\sum_{i=r-1}^{n-r} (-1)^r\binom{n+1-2r}{i+1-r},
$$
and the result follows by the binomial theorem. For $r=0$, we must take
into account the term $2c^n$ in front of the sum of terms of the form
$M(c^{i},c^{n-1-i})$ and we must also note the summation limits. We
obtain that the coefficient of $c^n$ in $\II(c^n)$ is
$$
2+\sum_{i=0}^{n-1} \binom{n+1}{i+1}=2+2^{n+1}-2=2^{n+1}.
$$
\end{proof}

\begin{theorem}
Assume that the finite vector  $(k_0,\ldots,k_{r})$ of nonnegative
integers satisfies $2r+k_0+k_2+\cdots+ k_{r}=n-1$. Then  
the coefficient of $c^{k_0}dc^{k_1} d\cdots c^{k_{r-1}}dc^{k_{r}}$ in
$\II(c^n)$, written as a $cd$-polynomial, is $2^{r+1}(k_0+1)(k_1+1)\cdots
(k_r+1)$. 
\end{theorem}  
\begin{proof}
We may obtain the $cd$-index by substituting $e^2=c^2-2d$ into the
$ce$-index. Hence the $cd$ word $c^{k_0}dc^{k_1} d\cdots
c^{k_{r-1}}dc^{k_{r}}$ is contributed by all $ce$-words that are obtained
from $c^{k_0}e^2c^{k_1} e^2\cdots c^{k_r}e^2c^{k_{r}}$ by replacing
some factors $c^2$ by $e^2$. By Proposition~\ref{prop:uce}, the coefficient
of $c^{k_0}e^2c^{k_1} e^2\cdots c^{k_r}e^2c^{k_{r}}$ in $\II(c^n)$ is
$(-1)^r\cdot 2^{n+1-2r}$, but when we replace all factors $e^2$ in this
word by $(c^2-2d)$, we have to multiply by $(-2)^r$. Hence, the
contribution of the term $(-1)^r\cdot 2^{n+1-2r}\cdot c^{k_0}e^2c^{k_1}
e^2\cdots c^{k_r}e^2c^{k_{r}}$ to the coefficient of $c^{k_0}dc^{k_1} d\cdots
c^{k_{r-1}}dc^{k_{r}}$ in $\II(c^n)$ is $2^{n+1-r}$.

When we replace any factor $c^2$ in $c^{k_0}e^2c^{k_1} e^2\cdots
c^{k_r}e^2c^{k_{r}}$ by $e^2$,
the coefficient of the resulting $ce$-word gets changed by a factor of
$(-2^{-2})=-1/4$. These additional factors $e^2$ contribute a factor of
$1$ when we replace $e^2$ with $(c^2-2d)$ and consider the coefficient
of $c^{k_0}dc^{k_1} d\cdots c^{k_{r-1}}dc^{k_{r}}$. To add up the
contribution of all these other terms consider first the special case
when we compute the coefficient $\gamma_n$ of $c^n$ in $\II(c^n)$, written as a
$cd$-polynomial. For example, for
$n=4$, by Proposition~\ref{prop:uce} we have
$$
\II(c^4)=2^5\cdot c^4-2^3\cdot (c^2e^2+ce^2c+e^2c^2)+2^1\cdot e^4,
$$
and if we rewrite this as a $cd$-polynomial, using $e^2=c^2-2d$, we
obtain
$$\gamma_4=2^5-2^3\cdot 3+2^1\cdot 1=10.$$
For general $n$ we obtain
$$
\gamma_n=\sum_{r=0}^{\lfloor\frac{n}{2}\rfloor}
(-1)^r\cdot 2^{n+1-2r}\cdot \binom{n-r}{n-2r} 
$$
It is easy to show (using for example the Fibonacci-type recurrence
$\gamma_n=2\gamma_{n-1}-\gamma_{n-2}$) that $\gamma_n=2(n+1)$.
Let us compute next the coefficient of $c^{k_0}dc^{k_1} d\cdots
c^{k_{r-1}}dc^{k_{r}}$ in $\II(c^n)$, written as a $cd$-monomial. As
noted above, the rewriting the term $(-1)^r\cdot 2^{n+1-2r}\cdot
c^{k_0}e^2c^{k_1} e^2\cdots c^{k_r}e^2c^{k_{r}}$ contributes 
$2^{n+1-r}$ to the coefficient of $c^{k_0}dc^{k_1} d\cdots
c^{k_{r-1}}dc^{k_{r}}$. To obtain the contribution of the other
$ce$-terms, we may repeat the above reasoning to each factor $c^{k_i}$
for $i=0,1,\ldots,r$. We obtain the coefficient
$$2^{n+1-r}\frac{\gamma_{k_0}}{2^{k_0+1}}\frac{\gamma_{k_1}}{2^{k_1+1}}\cdots 
    \frac{\gamma_{k_r}}{2^{k_1+1}}=2^{n+1-r}\frac{(k_0+1)(k_1+1)\cdots
    (k_r+1)}{2^{n-2r}} .$$
\end{proof}  

As we have seen in Proposition~\ref{prop:tbc}, the order complex of the
poset of intervals $\grI(P([1,n]))$ of the Boolean algebra $P([1,n])$
contains the type $B$ coxeter complex. Furthermore, the following
statement is well-known.

\begin{lemma}
The poset of intervals $\grI(P([1,n]))$ of the Boolean algebra $P([1,n])$
is isomorphic to the face lattice $C_n$ of the $n$-dimensional cube. 
\end{lemma}  
Indeed, we may identify each vertex $(x_1,\ldots,x_n)$ of the standard
cube $[0,1]^n$ with the subset $\sigma=\{i\in [1,n]\::\: x_i=1\}$ of
$[1,n]$. Each interval $[\sigma,\tau]$ corresponds to the face containing
all vertices $(x_1,\ldots,x_n)$ satisfying $x_i=0$ for $i\in [1,n]-\tau$
and $x_i=1$ for $i\in \sigma$.

The $cd$-index of the cubical lattice has been expressed by
Hetyei~\cite{Hetyei-andre} and by Ehrenborg and
Readdy~\cite{Ehrenborg-Readdy-rcubical} in terms of (different) signed
generalizations of {\em Andr\'e-permutations}. Andr\'e permutations,
first studied by Foata, Strehl and
Sch\"utzenberger~\cite{Foata-Schutzenberger,Foata-Strehl} were used by
Purtill~\cite{Purtill} to express the $cd$-index of the Boolean
algebra.

Purtill's approach may also be used to compute the interval transform of
the second kind $\II(P([1,n]))$ of a Boolean algebra, because of the
following observation: the set of all faces containing a vertex of an
$n$-dimensional hypercube, ordered by inclusion, form a lattice that is
isomorphic to the Boolean algebra $P([1,n])$. (In other words, the link
of a vertex in a cube is a simplex.) Hence we obtain
\begin{align}
\label{eq:eigenvector}  
\Psi_{\II(P([1,n]))}(c,d)=2^n\cdot \Psi_{P([1,n])}(c,d). 
\end{align}  

\begin{remark}
Equation~\ref{eq:eigenvector} exhibits a remarkable analogy to a result
of Ehrenborg and Readdy~\cite[Theorem 10.10]{Ehrenborg-Readdy-Tch}
completely describing all eigenvectors of the Tchebyshev transform of
the second kind, discussed in their paper. 
\end{remark}

\section{Concluding remarks}

It would be desirable to find more explicit formulas describing the
$cd$-index of a graded poset of intervals of an Eulerian poset, but
this seems harder than for the Tchebyshev transform studied 
in~\cite{Hetyei-tch}, \cite{Hetyei-mfp} and \cite{Ehrenborg-Readdy-Tch}. The
source of all difficulties seems that the operator $\iab$ recursively
``rotates'' the words involved: the recurrences call for cutting off certain
initial segment of some words and placing their reverse at the end. That
said, generalizations of permutohedra abound, and performing an
analogous sequence of stellar subdivisions on their duals, respectively
taking the graded poset of intervals for an associated poset may result
in interesting geometric constructions, producing perhaps new type $B$
analogues. A first step in this direction may be found in the work of
Athanasiadis~\cite{Athanasiadis-BE} where the $r$-fold edgewise
subdivision of the barycentric subdivision of a simplex is considered. 
Finally, applying the Tchebyshev transform studied
in~\cite{Hetyei-tch}, \cite{Hetyei-mfp} and \cite{Ehrenborg-Readdy-Tch} to
a Boolean algebra creates a poset whose order complex has the same
face numbers as the dual of a type $B$ permutohedron. It may be
interesting to find out whether the resulting polytope also has a nice
geometric representation.

\section*{Acknowledgments}
This work was partially supported by a grant from the Simons Foundation
(\#514648 to G\'abor Hetyei). Many thanks to the referees of an extended
abstract on this work, submitted to FPSAC 2020, for many helpful
suggestions. I am also indebted to Christos Athanasiadis for pointing
out the connection between Tchebyshev triangulations and second edgewise
subdivisions.

\end{document}

%% file: triangulations_pdf.tex
\begin{picture}(0,0)%
\includegraphics{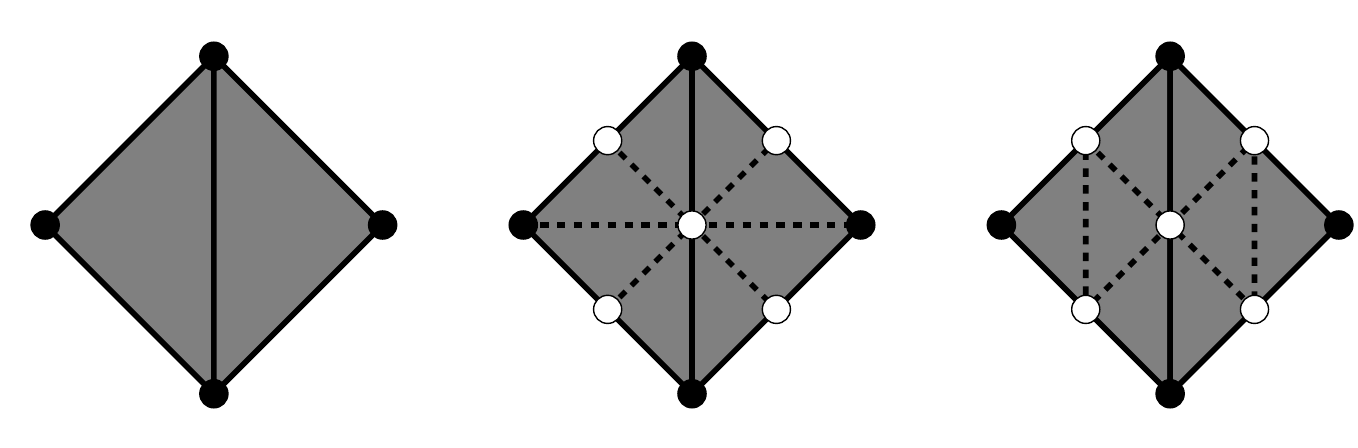}%
\end{picture}%
\setlength{\unitlength}{3552sp}%
\begingroup\makeatletter\ifx\SetFigFont\undefined%
\gdef\SetFigFont#1#2#3#4#5{%
  \reset@font\fontsize{#1}{#2pt}%
  \fontfamily{#3}\fontseries{#4}\fontshape{#5}%
  \selectfont}%
\fi\endgroup%
\begin{picture}(7223,2386)(1111,-3950)
\put(7276,-3886){\makebox(0,0)[lb]{\smash{{\SetFigFont{11}{13.2}{\familydefault}{\mddefault}{\updefault}{\color[rgb]{0,0,0}$v_1$}%
}}}}
\put(5626,-2611){\makebox(0,0)[lb]{\smash{{\SetFigFont{11}{13.2}{\familydefault}{\mddefault}{\updefault}{\color[rgb]{0,0,0}$v_4$}%
}}}}
\put(4726,-1711){\makebox(0,0)[lb]{\smash{{\SetFigFont{11}{13.2}{\familydefault}{\mddefault}{\updefault}{\color[rgb]{0,0,0}$v_2$}%
}}}}
\put(3676,-2611){\makebox(0,0)[lb]{\smash{{\SetFigFont{11}{13.2}{\familydefault}{\mddefault}{\updefault}{\color[rgb]{0,0,0}$v_3$}%
}}}}
\put(4726,-3886){\makebox(0,0)[lb]{\smash{{\SetFigFont{11}{13.2}{\familydefault}{\mddefault}{\updefault}{\color[rgb]{0,0,0}$v_1$}%
}}}}
\put(2176,-3886){\makebox(0,0)[lb]{\smash{{\SetFigFont{11}{13.2}{\familydefault}{\mddefault}{\updefault}{\color[rgb]{0,0,0}$v_1$}%
}}}}
\put(2176,-1711){\makebox(0,0)[lb]{\smash{{\SetFigFont{11}{13.2}{\familydefault}{\mddefault}{\updefault}{\color[rgb]{0,0,0}$v_2$}%
}}}}
\put(3076,-2611){\makebox(0,0)[lb]{\smash{{\SetFigFont{11}{13.2}{\familydefault}{\mddefault}{\updefault}{\color[rgb]{0,0,0}$v_4$}%
}}}}
\put(1126,-2611){\makebox(0,0)[lb]{\smash{{\SetFigFont{11}{13.2}{\familydefault}{\mddefault}{\updefault}{\color[rgb]{0,0,0}$v_3$}%
}}}}
\put(6226,-2611){\makebox(0,0)[lb]{\smash{{\SetFigFont{11}{13.2}{\familydefault}{\mddefault}{\updefault}{\color[rgb]{0,0,0}$v_3$}%
}}}}
\put(7276,-1711){\makebox(0,0)[lb]{\smash{{\SetFigFont{11}{13.2}{\familydefault}{\mddefault}{\updefault}{\color[rgb]{0,0,0}$v_2$}%
}}}}
\put(8176,-2611){\makebox(0,0)[lb]{\smash{{\SetFigFont{11}{13.2}{\familydefault}{\mddefault}{\updefault}{\color[rgb]{0,0,0}$v_4$}%
}}}}
\end{picture}%

%% file: poset_pdf.tex
\begin{picture}(0,0)%
\includegraphics{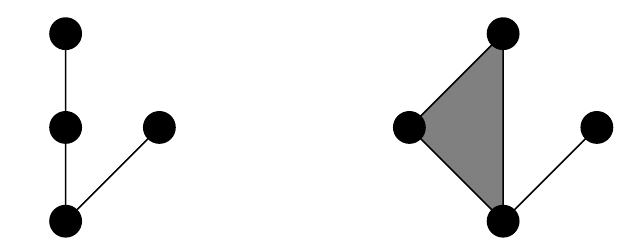}%
\end{picture}%
\setlength{\unitlength}{3947sp}%
\begingroup\makeatletter\ifx\SetFigFont\undefined%
\gdef\SetFigFont#1#2#3#4#5{%
  \reset@font\fontsize{#1}{#2pt}%
  \fontfamily{#3}\fontseries{#4}\fontshape{#5}%
  \selectfont}%
\fi\endgroup%
\begin{picture}(2955,1186)(3436,-2750)
\put(3451,-1711){\makebox(0,0)[lb]{\smash{{\SetFigFont{12}{14.4}{\familydefault}{\mddefault}{\updefault}{\color[rgb]{0,0,0}$u_3$}%
}}}}
\put(3451,-2161){\makebox(0,0)[lb]{\smash{{\SetFigFont{12}{14.4}{\familydefault}{\mddefault}{\updefault}{\color[rgb]{0,0,0}$u_2$}%
}}}}
\put(3451,-2686){\makebox(0,0)[lb]{\smash{{\SetFigFont{12}{14.4}{\familydefault}{\mddefault}{\updefault}{\color[rgb]{0,0,0}$u_1$}%
}}}}
\put(5551,-2686){\makebox(0,0)[lb]{\smash{{\SetFigFont{12}{14.4}{\familydefault}{\mddefault}{\updefault}{\color[rgb]{0,0,0}$u_1$}%
}}}}
\put(4276,-2161){\makebox(0,0)[lb]{\smash{{\SetFigFont{12}{14.4}{\familydefault}{\mddefault}{\updefault}{\color[rgb]{0,0,0}$u_4$}%
}}}}
\put(5101,-2161){\makebox(0,0)[lb]{\smash{{\SetFigFont{12}{14.4}{\familydefault}{\mddefault}{\updefault}{\color[rgb]{0,0,0}$u_2$}%
}}}}
\put(5551,-1711){\makebox(0,0)[lb]{\smash{{\SetFigFont{12}{14.4}{\familydefault}{\mddefault}{\updefault}{\color[rgb]{0,0,0}$u_3$}%
}}}}
\put(6376,-2161){\makebox(0,0)[lb]{\smash{{\SetFigFont{12}{14.4}{\familydefault}{\mddefault}{\updefault}{\color[rgb]{0,0,0}$u_4$}%
}}}}
\end{picture}%

%% file: posett_pdf.tex
\begin{picture}(0,0)%
\includegraphics{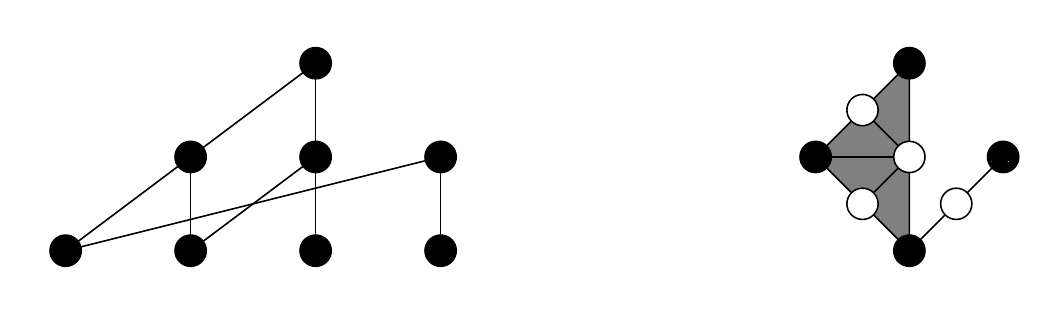}%
\end{picture}%
\setlength{\unitlength}{3947sp}%
\begingroup\makeatletter\ifx\SetFigFont\undefined%
\gdef\SetFigFont#1#2#3#4#5{%
  \reset@font\fontsize{#1}{#2pt}%
  \fontfamily{#3}\fontseries{#4}\fontshape{#5}%
  \selectfont}%
\fi\endgroup%
\begin{picture}(4980,1561)(1486,-2975)
\put(4726,-2236){\makebox(0,0)[lb]{\smash{{\SetFigFont{12}{14.4}{\familydefault}{\mddefault}{\updefault}{\color[rgb]{0,0,0}$[u_2,u_2]$}%
}}}}
\put(4951,-2461){\makebox(0,0)[lb]{\smash{{\SetFigFont{12}{14.4}{\familydefault}{\mddefault}{\updefault}{\color[rgb]{0,0,0}$[u_1,u_2]$}%
}}}}
\put(4951,-1936){\makebox(0,0)[lb]{\smash{{\SetFigFont{12}{14.4}{\familydefault}{\mddefault}{\updefault}{\color[rgb]{0,0,0}$[u_2,u_3]$}%
}}}}
\put(6226,-2461){\makebox(0,0)[lb]{\smash{{\SetFigFont{12}{14.4}{\familydefault}{\mddefault}{\updefault}{\color[rgb]{0,0,0}$[u_1,u_4]$}%
}}}}
\put(5926,-2011){\makebox(0,0)[lb]{\smash{{\SetFigFont{12}{14.4}{\familydefault}{\mddefault}{\updefault}{\color[rgb]{0,0,0}$[u_1,u_3]$}%
}}}}
\put(6451,-2236){\makebox(0,0)[lb]{\smash{{\SetFigFont{12}{14.4}{\familydefault}{\mddefault}{\updefault}{\color[rgb]{0,0,0}$[u_4,u_4]$}%
}}}}
\put(5551,-1561){\makebox(0,0)[lb]{\smash{{\SetFigFont{12}{14.4}{\familydefault}{\mddefault}{\updefault}{\color[rgb]{0,0,0}$[u_3,u_3]$}%
}}}}
\put(5551,-2836){\makebox(0,0)[lb]{\smash{{\SetFigFont{12}{14.4}{\familydefault}{\mddefault}{\updefault}{\color[rgb]{0,0,0}$[u_1,u_1]$}%
}}}}
\put(3751,-2236){\makebox(0,0)[lb]{\smash{{\SetFigFont{12}{14.4}{\familydefault}{\mddefault}{\updefault}{\color[rgb]{0,0,0}$[u_1,u_4]$}%
}}}}
\put(2476,-1561){\makebox(0,0)[lb]{\smash{{\SetFigFont{12}{14.4}{\familydefault}{\mddefault}{\updefault}{\color[rgb]{0,0,0}$[u_1,u_3]$}%
}}}}
\put(2776,-2911){\makebox(0,0)[lb]{\smash{{\SetFigFont{12}{14.4}{\familydefault}{\mddefault}{\updefault}{\color[rgb]{0,0,0}$[u_3,u_3]$}%
}}}}
\put(3376,-2911){\makebox(0,0)[lb]{\smash{{\SetFigFont{12}{14.4}{\familydefault}{\mddefault}{\updefault}{\color[rgb]{0,0,0}$[u_4,u_4]$}%
}}}}
\put(2101,-2911){\makebox(0,0)[lb]{\smash{{\SetFigFont{12}{14.4}{\familydefault}{\mddefault}{\updefault}{\color[rgb]{0,0,0}$[u_2,u_2]$}%
}}}}
\put(1501,-2911){\makebox(0,0)[lb]{\smash{{\SetFigFont{12}{14.4}{\familydefault}{\mddefault}{\updefault}{\color[rgb]{0,0,0}$[u_1,u_1]$}%
}}}}
\put(3001,-2011){\makebox(0,0)[lb]{\smash{{\SetFigFont{12}{14.4}{\familydefault}{\mddefault}{\updefault}{\color[rgb]{0,0,0}$[u_2,u_3]$}%
}}}}
\put(1801,-2011){\makebox(0,0)[lb]{\smash{{\SetFigFont{12}{14.4}{\familydefault}{\mddefault}{\updefault}{\color[rgb]{0,0,0}$[u_1,u_2]$}%
}}}}
\end{picture}%

%% file: posettpart_pdf.tex
\begin{picture}(0,0)%
\includegraphics{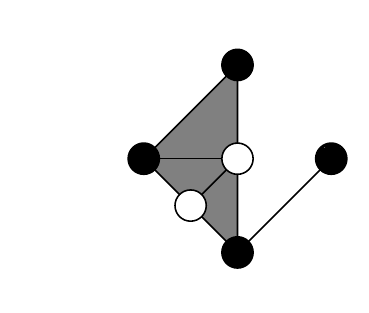}%
\end{picture}%
\setlength{\unitlength}{3947sp}%
\begingroup\makeatletter\ifx\SetFigFont\undefined%
\gdef\SetFigFont#1#2#3#4#5{%
  \reset@font\fontsize{#1}{#2pt}%
  \fontfamily{#3}\fontseries{#4}\fontshape{#5}%
  \selectfont}%
\fi\endgroup%
\begin{picture}(1755,1486)(4711,-2900)
\put(4726,-2236){\makebox(0,0)[lb]{\smash{{\SetFigFont{12}{14.4}{\familydefault}{\mddefault}{\updefault}{\color[rgb]{0,0,0}$[u_2,u_2]$}%
}}}}
\put(4951,-2461){\makebox(0,0)[lb]{\smash{{\SetFigFont{12}{14.4}{\familydefault}{\mddefault}{\updefault}{\color[rgb]{0,0,0}$[u_1,u_2]$}%
}}}}
\put(5926,-2011){\makebox(0,0)[lb]{\smash{{\SetFigFont{12}{14.4}{\familydefault}{\mddefault}{\updefault}{\color[rgb]{0,0,0}$[u_1,u_3]$}%
}}}}
\put(6451,-2236){\makebox(0,0)[lb]{\smash{{\SetFigFont{12}{14.4}{\familydefault}{\mddefault}{\updefault}{\color[rgb]{0,0,0}$[u_4,u_4]$}%
}}}}
\put(5551,-1561){\makebox(0,0)[lb]{\smash{{\SetFigFont{12}{14.4}{\familydefault}{\mddefault}{\updefault}{\color[rgb]{0,0,0}$[u_3,u_3]$}%
}}}}
\put(5551,-2836){\makebox(0,0)[lb]{\smash{{\SetFigFont{12}{14.4}{\familydefault}{\mddefault}{\updefault}{\color[rgb]{0,0,0}$[u_1,u_1]$}%
}}}}
\end{picture}%

%% file: gradedi_pdf.tex
\begin{picture}(0,0)%
\includegraphics{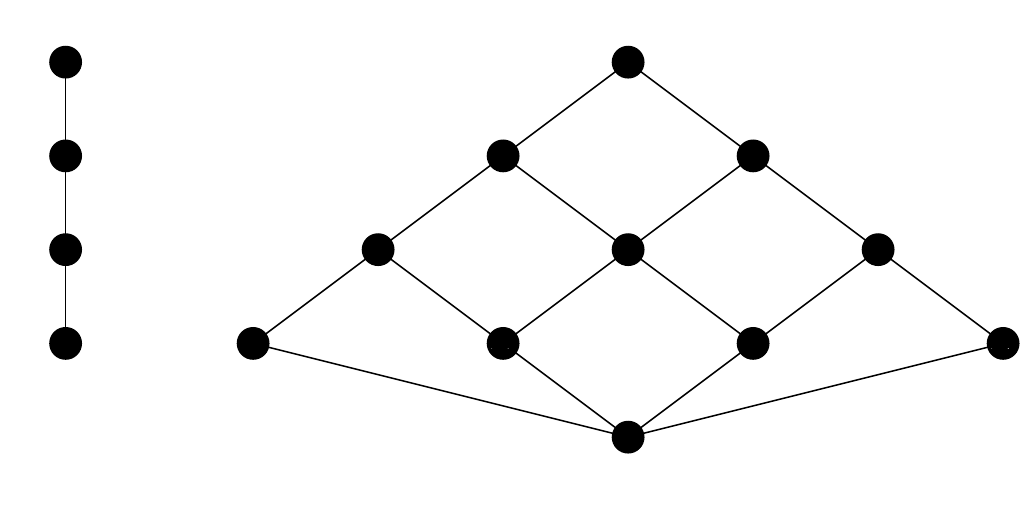}%
\end{picture}%
\setlength{\unitlength}{3947sp}%
\begingroup\makeatletter\ifx\SetFigFont\undefined%
\gdef\SetFigFont#1#2#3#4#5{%
  \reset@font\fontsize{#1}{#2pt}%
  \fontfamily{#3}\fontseries{#4}\fontshape{#5}%
  \selectfont}%
\fi\endgroup%
\begin{picture}(4905,2466)(3436,-3430)
\put(6376,-3361){\makebox(0,0)[lb]{\smash{{\SetFigFont{12}{14.4}{\familydefault}{\mddefault}{\updefault}{\color[rgb]{0,0,0}$\emptyset$}%
}}}}
\put(3451,-2686){\makebox(0,0)[lb]{\smash{{\SetFigFont{12}{14.4}{\familydefault}{\mddefault}{\updefault}{\color[rgb]{0,0,0}$\0$}%
}}}}
\put(3451,-2161){\makebox(0,0)[lb]{\smash{{\SetFigFont{12}{14.4}{\familydefault}{\mddefault}{\updefault}{\color[rgb]{0,0,0}$u_1$}%
}}}}
\put(3451,-1711){\makebox(0,0)[lb]{\smash{{\SetFigFont{12}{14.4}{\familydefault}{\mddefault}{\updefault}{\color[rgb]{0,0,0}$u_2$}%
}}}}
\put(3451,-1336){\makebox(0,0)[lb]{\smash{{\SetFigFont{12}{14.4}{\familydefault}{\mddefault}{\updefault}{\color[rgb]{0,0,0}$\1$}%
}}}}
\put(6226,-1111){\makebox(0,0)[lb]{\smash{{\SetFigFont{12}{14.4}{\familydefault}{\mddefault}{\updefault}{\color[rgb]{0,0,0}$[\0,\1]$}%
}}}}
\put(5701,-2236){\makebox(0,0)[lb]{\smash{{\SetFigFont{12}{14.4}{\familydefault}{\mddefault}{\updefault}{\color[rgb]{0,0,0}$[u_1,u_2]$}%
}}}}
\put(5251,-1711){\makebox(0,0)[lb]{\smash{{\SetFigFont{12}{14.4}{\familydefault}{\mddefault}{\updefault}{\color[rgb]{0,0,0}$[\0,u_2]$}%
}}}}
\put(7126,-1711){\makebox(0,0)[lb]{\smash{{\SetFigFont{12}{14.4}{\familydefault}{\mddefault}{\updefault}{\color[rgb]{0,0,0}$[u_1,\1]$}%
}}}}
\put(7726,-2161){\makebox(0,0)[lb]{\smash{{\SetFigFont{12}{14.4}{\familydefault}{\mddefault}{\updefault}{\color[rgb]{0,0,0}$[u_2,\1]$}%
}}}}
\put(8326,-2611){\makebox(0,0)[lb]{\smash{{\SetFigFont{12}{14.4}{\familydefault}{\mddefault}{\updefault}{\color[rgb]{0,0,0}$[\1,\1]$}%
}}}}
\put(4126,-2611){\makebox(0,0)[lb]{\smash{{\SetFigFont{12}{14.4}{\familydefault}{\mddefault}{\updefault}{\color[rgb]{0,0,0}$[\0,\0]$}%
}}}}
\put(7276,-2611){\makebox(0,0)[lb]{\smash{{\SetFigFont{12}{14.4}{\familydefault}{\mddefault}{\updefault}{\color[rgb]{0,0,0}$[u_2,u_2]$}%
}}}}
\put(4651,-2161){\makebox(0,0)[lb]{\smash{{\SetFigFont{12}{14.4}{\familydefault}{\mddefault}{\updefault}{\color[rgb]{0,0,0}$[\0,u_1]$}%
}}}}
\put(5026,-2611){\makebox(0,0)[lb]{\smash{{\SetFigFont{12}{14.4}{\familydefault}{\mddefault}{\updefault}{\color[rgb]{0,0,0}$[u_1,u_1]$}%
}}}}
\end{picture}%

%% file: typebdual_pdf.tex
\begin{picture}(0,0)%
\includegraphics{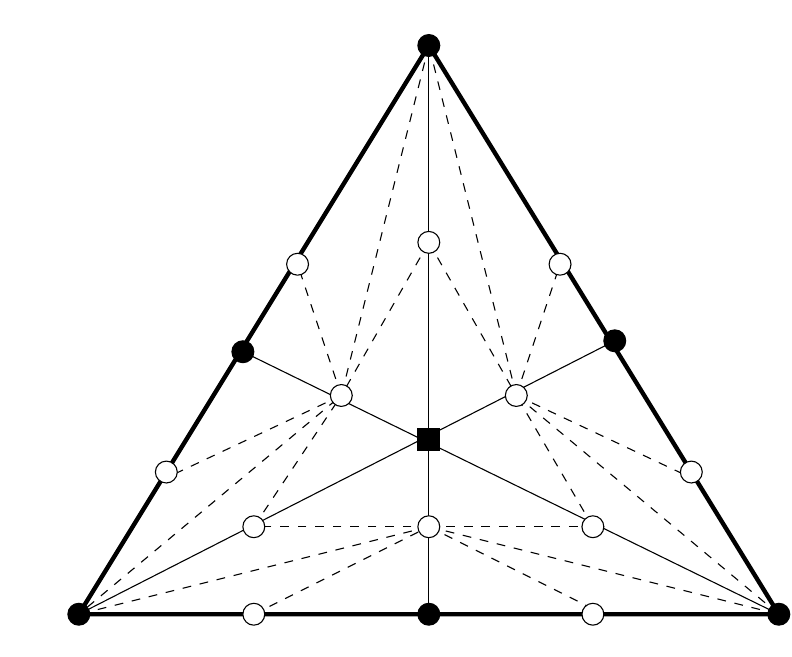}%
\end{picture}%
\setlength{\unitlength}{2763sp}%
\begingroup\makeatletter\ifx\SetFigFont\undefined%
\gdef\SetFigFont#1#2#3#4#5{%
  \reset@font\fontsize{#1}{#2pt}%
  \fontfamily{#3}\fontseries{#4}\fontshape{#5}%
  \selectfont}%
\fi\endgroup%
\begin{picture}(5423,4486)(4561,-4250)
\put(7051, 89){\makebox(0,0)[lb]{\smash{{\SetFigFont{8}{9.6}{\familydefault}{\mddefault}{\updefault}{\color[rgb]{0,0,0}$[\{3\}, \{3\}]$}%
}}}}
\put(8851,-2161){\makebox(0,0)[lb]{\smash{{\SetFigFont{8}{9.6}{\familydefault}{\mddefault}{\updefault}{\color[rgb]{0,0,0}$[\{2,3\}, \{2,3\}]$}%
}}}}
\put(4876,-2161){\makebox(0,0)[lb]{\smash{{\SetFigFont{8}{9.6}{\familydefault}{\mddefault}{\updefault}{\color[rgb]{0,0,0}$[\{1,3\}, \{1,3\}]$}%
}}}}
\put(4576,-4186){\makebox(0,0)[lb]{\smash{{\SetFigFont{8}{9.6}{\familydefault}{\mddefault}{\updefault}{\color[rgb]{0,0,0}$[\{1\}, \{1\}]$}%
}}}}
\put(9601,-4186){\makebox(0,0)[lb]{\smash{{\SetFigFont{8}{9.6}{\familydefault}{\mddefault}{\updefault}{\color[rgb]{0,0,0}$[\{2\}, \{2\}]$}%
}}}}
\put(7576,-3286){\makebox(0,0)[lb]{\smash{{\SetFigFont{8}{9.6}{\familydefault}{\mddefault}{\updefault}{\color[rgb]{0,0,0}$[\emptyset, \{1,2\}]$}%
}}}}
\put(6976,-4186){\makebox(0,0)[lb]{\smash{{\SetFigFont{8}{9.6}{\familydefault}{\mddefault}{\updefault}{\color[rgb]{0,0,0}$[\{1,2\}, \{1,2\}]$}%
}}}}
\put(5851,-4186){\makebox(0,0)[lb]{\smash{{\SetFigFont{8}{9.6}{\familydefault}{\mddefault}{\updefault}{\color[rgb]{0,0,0}$[\{1\}, \{1,2\}]$}%
}}}}
\put(8326,-4186){\makebox(0,0)[lb]{\smash{{\SetFigFont{8}{9.6}{\familydefault}{\mddefault}{\updefault}{\color[rgb]{0,0,0}$[\{2\}, \{1,2\}]$}%
}}}}
\put(7651,-2836){\makebox(0,0)[lb]{\smash{{\SetFigFont{8}{9.6}{\familydefault}{\mddefault}{\updefault}{\color[rgb]{0,0,0}$\emptyset$}%
}}}}
\end{picture}%

%% file: butterfly_pdf.tex
\begin{picture}(0,0)%
\includegraphics{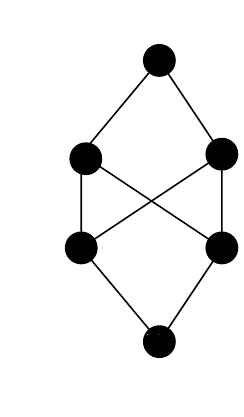}%
\end{picture}%
\setlength{\unitlength}{3947sp}%
\begingroup\makeatletter\ifx\SetFigFont\undefined%
\gdef\SetFigFont#1#2#3#4#5{%
  \reset@font\fontsize{#1}{#2pt}%
  \fontfamily{#3}\fontseries{#4}\fontshape{#5}%
  \selectfont}%
\fi\endgroup%
\begin{picture}(1155,1995)(2986,-2971)
\put(3676,-2911){\makebox(0,0)[lb]{\smash{{\SetFigFont{12}{14.4}{\familydefault}{\mddefault}{\updefault}{\color[rgb]{0,0,0}$0$}%
}}}}
\put(3001,-2236){\makebox(0,0)[lb]{\smash{{\SetFigFont{12}{14.4}{\familydefault}{\mddefault}{\updefault}{\color[rgb]{0,0,0}$-1$}%
}}}}
\put(3001,-1786){\makebox(0,0)[lb]{\smash{{\SetFigFont{12}{14.4}{\familydefault}{\mddefault}{\updefault}{\color[rgb]{0,0,0}$-2$}%
}}}}
\put(4126,-2236){\makebox(0,0)[lb]{\smash{{\SetFigFont{12}{14.4}{\familydefault}{\mddefault}{\updefault}{\color[rgb]{0,0,0}$1$}%
}}}}
\put(4126,-1786){\makebox(0,0)[lb]{\smash{{\SetFigFont{12}{14.4}{\familydefault}{\mddefault}{\updefault}{\color[rgb]{0,0,0}$2$}%
}}}}
\put(3676,-1111){\makebox(0,0)[lb]{\smash{{\SetFigFont{12}{14.4}{\familydefault}{\mddefault}{\updefault}{\color[rgb]{0,0,0}$3$}%
}}}}
\end{picture}%

%% file: btch-v02.bbl
\begin{thebibliography}{99}

\bibitem{Anwar}
I.\ Anwar and S.\ Nazir, 
The $f$- and $h$-vectors of interval subdivisions,
{\it J.\ Combin.\ Theory Ser.\ A } {\bf 169} (2020), 105124, 22 pp.   
  
\bibitem{Athanasiadis}  
C.\ A.\ Athanasiadis, 
On the M\"obius function of a lower Eulerian Cohen-Macaulay poset,
{\it J.\ Algebraic Combin.\ } {\bf 35} (2012), 373--388. 

\bibitem{Athanasiadis-survey}  
C.\ A.\ Athanasiadis, 
A survey of subdivisions and local h-vectors, In: The mathematical
legacy of Richard P.\ Stanley, Patricia Hersh, Thomas Lam, Pavlo
Pylyavskyy and Victor Reiner Editors, 39--51, Amer.\ Math.\ Soc.,
Providence, RI, 2016.

\bibitem{Athanasiadis-BE}
C.\ A.\ Athanasiadis, 
Binomial Eulerian polynomials for colored permutations,
{\it J.\ Combin.\ Theory Ser.\ A} {\bf 173} (2020), 105214, 38 pp. 
  
\bibitem{Athanasiadis-Savvidou}
  C.\ A.\ Athanasiadis and C.\ Savvidou,
A symmetric unimodal decomposition of the derangement polynomial of type
$B$, preprint 2013, arXiv:1303.2302 [math.CO].


  
\bibitem{Bayer-Billera}
M.\ Bayer, and L.J.\ Billera, 
Generalized Dehn-Sommerville relations for polytopes, spheres and
Eulerian partially ordered sets, 
{\it Invent.\ Math.\ } {\bf 79} (1985), 143--157. 

  
\bibitem{Bayer-Klapper}
M.M.\ Bayer and A.\ Klapper, 
A new index for polytopes,
{\it Discrete Comput.\ Geom.\ } {\bf 6} (1991), 33--47.   

\bibitem{Brenti-Welker}
F.\ Brenti, V.\ Welker, 
The Veronese construction for formal power series and graded algebras,
{\it Adv.\ in Appl.\ Math.\ } {\bf 42} (2009), 545--556.   

\bibitem{Ehrenborg-Fox}
  R.\ Ehrenborg and H.\ Fox,
Inequalities for cd-indices of joins and products of polytope, 
{\it Combinatorica} {\bf 23} (2003), 427--452.   

\bibitem{Ehrenborg-Readdy-cop}
  R.\ Ehrenborg and M.\ Readdy,
Coproducts and the cd-index. (English summary)
{\it J.\ Algebraic Combin.\ } {\bf 8} (1998), 273--299. 

\bibitem{Ehrenborg-Readdy-rcubical}
 R.\ Ehrenborg and M.\ Readdy,
The ${\bf r}$-cubical lattice and a generalization of the
${\bf cd}$-index, 
{\it European J.\ Combin.\ } {\bf 17} (1996), 709--725. 

 
\bibitem{Ehrenborg-Readdy-Tch}
  R.\ Ehrenborg and M.\ Readdy,
The Tchebyshev transforms of the first and second kind,
{\it Ann.\ Comb.\ } {\bf 14} (2010), 211--244.   

\bibitem{Fomin-Reading}
S.\ Fomin, and N.\ Reading, 
Root systems and generalized associahedra, in: Geometric combinatorics, 63--131,
IAS/Park City Math.\ Ser., 13, Amer. Math. Soc., Providence, RI, 2007.   

\bibitem{Foata-Schutzenberger} D.\ Foata and M. P.\ Sch\"{u}tzenberger,
  Nombres d'Euler et permutations alternantes,
In: J.N. Srivastava et al.,
A Survey of Combinatorial Theory, Amsterdam, North--Holland,
1973, pp.\ 173--187.

\bibitem{Foata-Strehl} D.\ Foata and V.\ Strehl, Rearrangements of the
symmetric group and enumerative properties of the tangent and secant
numbers, {\it Math. Z.} {\bf 137} (1974), 257--264.

\bibitem{Fox}
N.\ B.\ Fox, 
A lattice path interpretation of the diamond product,
{\it Ann.\ Comb.\ } {\bf 20} (2016), 569--586.   

\bibitem{Freudenthal}
Freudenthal, Hans
Simplizialzerlegungen von beschr\"ankter Flachheit,
{\it Ann.\ of Math.\  (2)} {\bf 43} (1942), 580--582.   

\bibitem{Hetyei-andre}
G.\ Hetyei,
On the cd-variation polynomials of Andr\'e and Simsun permutations,
{\it Discrete Comput.\ Geom.\ } {\bf 16} (1996), 259--275. 
  
\bibitem{Hetyei-tch}
  G.\ Hetyei,
Tchebyshev posets,
{\it Discrete Comput.\ Geom.\ } {\bf 32} (2004), 493--520.   

\bibitem{Hetyei-mfp}
G.\ Hetyei,
Matrices of formal power series associated to binomial posets,
{\it J.\ Algebraic Combin.\ } {\bf 22} (2005), 65--104.   

\bibitem{Hetyei-tt}
  G.\ Hetyei,
Tchebyshev triangulations of stable simplicial complexes,
{\it J.\ Combin.\ Theory Ser.\ A} {\bf 115} (2008), 569--592.   

\bibitem{Hetyei-Nevo}
  G.\ Hetyei and E.\ Nevo,
Generalized Tchebyshev triangulations,
{\it J.\ Combin.\ Theory Ser.\ A} {\bf 137} (2016), 88--125.   

\bibitem{Hohlweg-Lange}
C.\ Hohlweg and C.~E.~M.~C.~Lange, 
Realizations of the associahedron and cyclohedron,
{\it Discrete Comput.\ Geom.\ } {\bf 37 }(2007), 517--543.   

\bibitem{Jojic}
D. Joji\'{c}, 
The cd-index of the poset of intervals and $E_t$-construction,
{\it Rocky Mountain J.\ Math.\ } {\bf 40} (2010), 527--541.   


\bibitem{Loday}
J-L.\ Loday, 
Realization of the Stasheff polytope,
{\it Arch.\ Math.\ (Basel)} {\bf 83} (2004), 267--278.   

\bibitem{Ma}
Shi-Mei Ma,
A family of two-variable derivative polynomials for tangent and secant,
{\it Electron.\ J.\ Combin.\ } {\bf 20} (2013), no. 1, Paper 11, 12 pp.   

\bibitem{OEIS}
OEIS Foundation Inc.\ (2011), ``The On-Line Encyclopedia of Integer Sequences,''
published electronically at \url{http://oeis.org}. 

\bibitem{Postnikov}
A.\ Postnikov,
Permutohedra, associahedra, and beyond,
{\it Int.\ Math.\ Res.\ Not.\ } IMRN 2009, 1026--1106.

\bibitem{Purtill}
  M.\ Purtill,
Andr\'e permutations, lexicographic shellability and the $cd$-index of a
convex polytope, 
{\it Trans.\ Amer.\ Math.\ Soc.\ } {\bf 338} (1993), 77--104. 
  

\bibitem{Stanley-EC1}
R.P.\ Stanley, 
``Enumerative combinatorics, Volume 1,''
Second edition. Cambridge Studies in Advanced Mathematics, 49. Cambridge
University Press, Cambridge, 2012. 

\bibitem{Stanley-flag}
R.P.\ Stanley, 
Flag f-vectors and the cd-index,
{\it Math.\ Z.\ } {\bf 216} (1994), 483--499.   

\bibitem{Stanley-subdivisions}
R.P.\ Stanley,
Subdivisions and local $h$-vectors,
{\it J.\ Amer.\ Math.\ Soc.\ } {\bf 5} (1992), 805--851.   

\bibitem{Wachs-pt}

M.L.\ Wachs, 
Poset topology: tools and applications, in: Geometric combinatorics, 497--615,
IAS/Park City Math.\ Ser., 13, Amer.\ Math.\ Soc., Providence, RI, 2007.

\bibitem{Walker}
J.\ W.\ Walker,
Canonical homeomorphisms of posets,
{\it European J.\ Combin.\ } {\bf 9} (1988), 97--107. 

\end{thebibliography}
